\pdfoutput=1
\documentclass[11pt, a4paper, oneside]{amsart}
\usepackage[english]{babel}
\usepackage[T1]{fontenc}
\usepackage{lmodern}

\usepackage[usenames,dvipsnames]{xcolor}
\usepackage{svg}

\usepackage[colorlinks=true,linkcolor=NavyBlue,urlcolor=RoyalBlue,citecolor=PineGreen,
	hypertexnames=false]{hyperref}

\usepackage{cleveref}

\usepackage{amsthm}
\usepackage{amsmath}
\usepackage{amssymb}
\usepackage{amsfonts}
\usepackage{mathtools}
\usepackage{stmaryrd}
\usepackage{bm}
\usepackage{mathbbol}
\usepackage{shuffle}

\usepackage{tikz}
\usepackage{tikz-cd} \usepackage{subfigure}
\usepackage{verbatim}
\usepackage[margin=0.7in]{geometry}

\usepackage[font=small]{caption}
\usepackage{adjustbox}

\usepackage{enumerate}

\usepackage[linecolor=white,backgroundcolor=white,bordercolor=white,textsize=tiny]{todonotes}

\makeatletter
\renewcommand{\tocsection}[3]{%
	\indentlabel{\@ifnotempty{#2}{\bfseries\ignorespaces#1 #2.\,\,}}\bfseries#3}
\renewcommand{\tocsubsection}[3]{%
	\indentlabel{\@ifnotempty{#2}{\ignorespaces#1 #2\quad}}#3}
\renewcommand{\tocsubsubsection}[3]{%
	\quad\quad\quad\indentlabel{\@ifnotempty{#2}{\ignorespaces#1 #2\quad}}#3}

\newcommand\@dotsep{4.5}
\def\@tocline#1#2#3#4#5#6#7{\relax
	\ifnum #1>\c@tocdepth 
	\else
		\par \addpenalty\@secpenalty\addvspace{#2}%
		\begingroup \hyphenpenalty\@M
		\@ifempty{#4}{%
			\@tempdima\csname r@tocindent\number#1\endcsname\relax
		}{%
			\@tempdima#4\relax
		}%
		\parindent\z@ \leftskip#3\relax \advance\leftskip\@tempdima\relax
		\rightskip\@pnumwidth plus1em \parfillskip-\@pnumwidth
		#5\leavevmode\hskip-\@tempdima{#6}\nobreak
		\leaders\hbox{$\m@th\mkern \@dotsep mu\hbox{.}\mkern \@dotsep mu$}\hfill
		\nobreak
		\hbox to\@pnumwidth{\@tocpagenum{\ifnum#1=1\bfseries\fi#7}}\par
		\nobreak
		\endgroup
	\fi}
\AtBeginDocument{%
	\expandafter\renewcommand\csname r@tocindent0\endcsname{0pt}
}
\newcommand\blfootnote[1]{%
	\begingroup
	\renewcommand\thefootnote{}\footnote{#1}%
	\addtocounter{footnote}{-1}%
	\endgroup
}
\def\l@subsection{\@tocline{2}{0pt}{2.5pc}{5pc}{}}
\makeatother

\DeclareFontFamily{U}{MnSymbolC}{}
\DeclareFontShape{U}{MnSymbolC}{m}{n}{
	<-5.5> MnSymbolC5
	<5.5-6.5> MnSymbolC6
	<6.5-7.5> MnSymbolC7
	<7.5-8.5> MnSymbolC8
	<8.5-9.5> MnSymbolC9
	<9.5-11.5> MnSymbolC10
	<11.5-> MnSymbolCb12
}{}

\let\todon\todo
\renewcommand{\todo}[1]{\todon{\color{red}#1}}

\makeatletter
\providecommand{\leftsquigarrow}{%
  \mathrel{\mathpalette\reflect@squig\relax}%
}
\newcommand{\reflect@squig}[2]{%
  \reflectbox{$\m@th#1\rightsquigarrow$}%
}
\makeatother

\newcommand{\sbt}{\,\begin{tikzpicture}[baseline=(X.base)]%
		\node[draw, fill,black,circle, inner sep=1pt] at (0,0.1) {};
		\node[circle,inner sep=0pt,outer sep=0pt] (X){$\ $};
	\end{tikzpicture}%
	\,
}

\newtheoremstyle{bfnote}%
{}{}%
{\itshape}{}%
{\bfseries}{.}%
{ }%
{\thmname{#1}\thmnumber{ #2}\thmnote{ (#3)}}

\theoremstyle{bfnote}
\newtheorem{theorem}{Theorem}[section]
\newtheorem{example}[theorem]{Example}
\newtheorem{proposition}[theorem]{Proposition}
\newtheorem*{proposition*}{Proposition}
\newtheorem{corollaire}{Corollary}
\newtheorem*{theorem*}{Theorem}
\newtheorem{lemma}[theorem]{Lemma}
\newtheorem{remark}[theorem]{Remark}
\newtheorem{definition}[theorem]{Definition}
\newtheorem{definitionproposition}[theorem]{Definition-Proposition}

\theoremstyle{plain}

\newcommand{\ncl}{ {\mathrm{NCL}}}

\newcommand{\nc}{ {\mathrm{NC}}}
\newcommand{\nco}{{\mathcal{N}\mathcal{C}}}
\newcommand{\mult}{ { \mathrm{Mult} }}

\title[Operator-valued $T$-transform]{On the twisted factorization of the $T$-transform}

\author{Kurusch Ebrahimi-Fard, Nicolas Gilliers}
\address{Department of Mathematical Sciences, Norwegian University of Science and Technology (NTNU), 7491 Trondheim, Norway.}

\begin{document}




\begin{abstract}
    The amalgamated $T$-transform of a non-commutative distribution was introduced by K.~Dykema. It provides a fundamental tool for computing distributions of random variables in Voiculescu's free probability theory. The $T$-transform factorizes in a rather non-trivial way over a product of free random variables. In this article, we present a simple graphical proof of this property, followed by a more conceptual one, using the abstract setting of an operad with multiplication.
\end{abstract}
\maketitle
\blfootnote{Date: \today}

\tableofcontents

\def\fs{{\mathbb{C}[[\mathcal{P}]]}}
\def\grc{{G_{\ensuremath\raisebox{2pt}{$\scriptstyle\centerdot$}}^{inv}}}%

\section{Introduction}
\label{sec:intro}

In reference \cite{dykema2007multilinear}, K.~Dykema introduced and studied two central objects in free probability theory, i.e., the operator-valued $R$-transform, more precisely, the \emph{unsymmetrised} $R$-transform, as well as the -- interrelated -- operator-valued \emph{unsymmetrised} $S$- and $T$-transforms. Those transforms play a fundamental role in both scalar- as well as operator-valued free probability theory \cite{mingo2017free,nica2006lectures} as they allow for the effective --algorithmic-- calculation of the distribution of a sum respectively product of free random variables. In the scalar-valued case, they can be traced back to the seminal works by Voiculescu \cite{voiculescu1986addition,voiculescu1987multiplication}. Here, the $R$- and $T$-transforms with respect to a random variable $a$ in a non-commutative probability space $(\mathcal{A}, \phi)$ are formal power series in one variable, $R_a(z), T_a(z) \in \mathbb{K}\langle\langle z \rangle\rangle$. If $a$ and $b$ are two free random variables in $\mathcal{A}$, with $\phi(a)=\phi(b)=1$, then these transforms are linear
\begin{equation}
\label{scalarR}
	R_{a+b}(z) = R_{a}(z) + R_{b}(z)
\end{equation}
respectively multiplicative
\begin{equation}
\label{scalarT}
	T_{ab}(z) = T_{a}(z) \cdot T_{b}(z).
\end{equation}
The product on the right-hand side of \eqref{scalarT}	is the Cauchy product, defined for two series $f(z),g(z) \in \mathbb{K}\langle\langle z \rangle\rangle$ by:
$$
(f\cdot g)_{n} = \sum_{\substack{k,q \geq 0, \\ k+q = n}} f_{k}g_{q}.
$$
The $S$-transform is defined as the inverse (with respect to the Cauchy product) of the $T$-transform, $S_{a}(z)=T_{a}^{-1}(z)$. All three transforms are related through the \emph{distribution series} $\Phi_{a}(z)\in \mathbb{K}\langle\langle z \rangle\rangle$, $a \in \mathcal{A}$,
\begin{equation}
\label{scalardistrib}
	\Phi_{a}(z) = \sum_{n\geq 1} \phi(a^{n})z^{n-1},
\end{equation}
namely \cite{dykema2007multilinear},
\allowdisplaybreaks
\begin{equation}
\label{eqn:defrt}
\begin{split}
	 & (z + z^2 \Phi_{a}(z))^{\langle -1\rangle} = z(1 + z R_{a} (z))^{-1}  \\
	 & (z\Phi_{a}(z))^{\langle -1 \rangle} = z(1 + z)^{-1} S_{a}(z).
\end{split}
\end{equation}
The inverses on the left-hand sides of the above equations are computed with respect to \emph{composition} of formal power series, here denoted by $\circ$ and defined on $ z\mathbb{K}\langle\langle z \rangle\rangle$ by the well-known formula
\begin{equation}
	(f\circ g)_{n} = \sum_{\substack{k,n_{1},\ldots,n_{k} \geq 1 \\ n_{1} + \cdots +n_{k}= n}}f_{k}g_{n_{1}} \cdots g_{n_{k}}.
\end{equation}

\medskip

The so-called free additive and multiplicative convolution problems have been shown by Voiculescu to admit solutions by constructing canonical random variables. Dykema verified in \cite{dykema2007multilinear} that in the operator-valued case, these results admit counterparts, where so-called  \emph{multilinear function series} play the role of formal power series.

More specifically, let $(\mathcal{A},\phi,B)$ be an operator-valued probability space, that is, $\mathcal{A}$ is an unital algebra with unital subalgebra $B \subset \mathcal{A}$ and conditional expectation $\phi: \mathcal{A} \to B$, $\phi(b)=b$ and $\phi(b_1 a  b_2)=b_1\phi(a)b_2$ for all $b,b_1,b_2 \in B$ and $a \in \mathcal{A}$ \cite{mingo2017free}. Note that further below (Definition \ref{Def:OVPS}) we will work with a slightly more general notion of operator-valued probability space. Dykema introduced the notion of multilinear function series, that is, the data of a sequence $(\alpha_{n} : B^{\otimes n} \rightarrow B)_{n\geq 0}\in \mathrm{Mult}[[B]]$ of multilinear maps on the algebra $B$. We adopt the convention that $B^{\otimes 0}=\mathbb{C}$. This implies that $\alpha_0$ can be considered as an element of $B$. Let $a\in \mathcal{A}$ be a random variable with $\phi(a)=1_B(=1_\mathcal{A})$. The multilinear function series $\Phi_{a}=(\Phi_{a,0},\Phi_{a,1},\ldots) \in \mathrm{Mult}[[B]]$ that replaces the scalar-valued distribution series \eqref{scalardistrib} is defined through $\Phi_{a,0}=\phi(a)$ and for $n>0$
\begin{equation}
	\label{eqn:opvalueddistserie}
	\Phi_{a,n} (b_{1},\ldots,b_{n}) = \phi(ab_{1}ab_{2} \cdots ab_{n}a).
\end{equation}
Given two multilinear function series $\alpha, \beta \in \mathrm{Mult}[[B]]$, we define their formal product by
\begin{equation}
\label{multiplication}
	(\alpha\centerdot\beta)_{n}(b_{1},\ldots, b_{n}) = \sum_{k=0}^{n}\alpha_{k}(b_{1},\ldots, b_{k})\beta_{n-k}(b_{k+1},\ldots, b_{n}).
\end{equation}
This product turns the space $\mathrm{Mult}[[B]]$ into an unital algebra, with unit $1 = (1_{B},0,\ldots) \in \mathrm{Mult}[[B]]$. We define a second product on $\mathrm{Mult}[[B]]$ given by  \emph{composition}, denoted $\alpha \circ \beta$, if $\beta_{0} = 0$, by $(\alpha \circ \beta)_{0} = \alpha_{0}$, and
\begin{equation*}
	(\alpha\circ\beta)_{n}(b_{1},\ldots,b_{n}) =
	\sum_{k=1}^{n}
	\sum_{\substack{p_{1},\ldots,p_{k}
	\geq 1 \\p_{1}+\cdots+p_{k}=n}}
	\alpha_{k}\Big(\beta_{p_{1}}(b_{1},\ldots, b_{p_{1}}),\ldots, \beta_{p_{k}}(b_{q_{k}+1},\ldots, b_{q_{k}+p_{k}})\Big),
\end{equation*}
where $q_{j} := p_{1} + \cdots + p_{j-1}$. Notice that composition is linear only in the left argument. The unit for this product is the multilinear function series
\begin{equation}
\label{compId}
    I = (0, \mathrm{id}_{B}, 0, \ldots).
\end{equation}
Following an analogue approach as in \eqref{eqn:defrt} using  $\Phi_{a}$  defined by \eqref{eqn:opvalueddistserie}, one obtains operator-valued counterparts of the $S$-, $R$- and $T$-transforms that are now multilinear function series. By constructing random variables with prescribed $T$-transform, Dykema showed in \cite{dykema2007multilinear} that the multiplicativity of the $T$-transform \eqref{scalarT} in the scalar-valued case generalises to the following so-called twisted factorization in the operator-valued case:
\begin{equation}
	\label{eqn:twisted}
	T_{ab} = \big(T_{a}\circ (T_{b}\centerdot I \centerdot T_{b}^{-1})\big)\centerdot T_{b}.
\end{equation}
Note the appearance of the product defined in  \eqref{multiplication}.
\def\ncldc{{{\rm NCL}_{D}^{(1)}}}

R.~Speicher showed in \cite{speicher1994multiplicative} that non-crossing set partitions underlie the combinatorics of moment-cumulant relations. To relate free cumulants to the coefficients of the $T$-transform,  the same role is played by so-called \emph{non-crossing linked set partitions}. In particular, M\"obius inversion on a certain poset $\ncldc$ of \emph{connected non-crossing linked set partitions} relates coefficients of the $T$-transform to free cumulants,
\begin{equation}
\label{eqn:cumulantttransform}
\kappa_{n}(ab_{1},\ldots,ab_{n-1},a) = \sum_{\pi \in \mathrm{NCL}_{D}^{(1)}(n)} t_{a}(\pi)(b_{1},\ldots,b_{n-1}).
\end{equation}
For small values of $n$, one has
\begin{align*}
\kappa_{2}(ab_{1}, a) &= t_{a}(1)(b_{1})\\
\kappa_{3}(ab_{1},ab_{2},a) &= t_{a}(2)(b_{1},b_{2}) + t_{a}(1)(b_{1}t_{a}(1)(b_{2}))\\
\kappa_{4}(ab_{1},ab_{2},ab_{3},a) &= t_{a}(3)(b_{1},b_{2},b_{3}) + t_{a}(2)(b_{1}t_{a}(1)(b_{2}), b_{3}) \\
&\quad +t_{a}(2)(b_{1},b_{2}t_{a}(1)(b_{3})) + t_{a}(1)(b_{1}t_{a}(2)(b_{2},b_{3})).
\end{align*}
Here, $t_{a}(1_n)=t_{a}(n)$. We refer to \cite{dykema2007multilinear} for the precise definition of the multilinear map $t_{a}(\pi)$. Let us just mention that $t_a(\pi)$ factorises in a certain way over the blocks of a non-crossing linked partition. In this article, this factorization is expressed in terms of a certain morphism, whose range comprises the values $t_a(\pi),\pi\in \mathrm{NCL}_{D}^{(1)}$, and which is compatible with an operadic composition on non-crossing linked partitions.
The above relations (one for each integer $n\geq 1$) can be cast into a fixed point equation on multilinear function series:
$$
K_{a} = 1+\sum_{n\geq 1}K_a^{(n)},\quad K_a^{(n)}(b_{1},\ldots,b_{n}) = \kappa_{n+1}(ab_{1},ab_{2},\ldots,ab_{n},a).
$$
The relations displayed in \eqref{eqn:cumulantttransform} are equivalent to $$
K_{a} = T_{a} \circ (I\centerdot K_{a}),
$$
where $I$ is the identity for composition \eqref{compId} in $\mathrm{Mult}[[B]]$.

\medskip


In the first part of this work we will give a different (and shorter) proof of equation \eqref{eqn:twisted} which is operadic in nature. It does not involve the construction of canonical random variables and therefore yields to a more conceptual understanding of equation \eqref{eqn:twisted} as resulting from distributivity of the composition product $\circ$ over the product $\centerdot$.

This distributivity together with the specific form of the composition product $\circ$ are constitutive to the notion of Gerstenhaber algebra. Examples of such type of algebras  are obtained by considering operads equipped with a distinguished operator of arity two that we call \emph{multiplication} (for reasons explained below). This multiplication endows the set of operators with a monoidal law, over which the operadic product distributes. The operad with multiplication used in our work is the endomorphism operad of $B$ spanned by multilinear maps over the algebra $B$. The multiplication coincides with the algebra product on $B$.

\begin{remark}
It is worth noticing that these algebras admit a homotopical version \cite{gerstenhaber1994homotopy}. In this context, the operadic multiplication can be used to define a differential on the underlying collection of the operad. This construction yields for example the Hochschild cohomology of $B$, if one starts with the endomorphism operad of $B$.
\end{remark}

\def\gr{{G^{inv}}}
\def\grm{{G^{inv}_m}}

In the second part of this work, we take a leap in abstraction and introduce in the setting of an operad $\mathcal{P}$ with multiplication $m$ a free product on the set  $G^{inv}$ of formal series of operators  -- with non-zero constant coefficient. The latter is supporting a monoidal structure stemming from the multiplication $m$. We define in this context the $T$-transform.

This permits us to give another, in some sense more fundamental proof of the twisted factorization of the abstract $T$-transform. In particular, we highlight the role of left and right translations by the identity of the operad. These maps provide two different injections of the set $G^{inv}$ into the \emph{diffeomorphism group} $G$ of the operad $\mathcal{P}$ (certain series of operators on $\mathcal{P}$) denoted $\rho$ and $\lambda$.

These translations are not group morphisms. However, they are injective and their inverses are \emph{cocycles with respect to a left action} $\curvearrowright$ \emph{of the diffeomorphism group $G$ on $G^{inv}$}. This algebraic setup is well understood. See for example \cite{frabetti}. We consider in addition a right action $\curvearrowleft$ which is the conjugation action of the group $(G^{inv},\centerdot)$ restricted to $G$. This permits to relate the translations $\rho$ to $\lambda$.
Ultimately, the twisted factorization of the abstract $T$-transform is implied by
\begin{enumerate}
    \item Distributivity of the left action $\curvearrowright$ of the group $G$ over the product $\centerdot$,
    \item The cocycle property of $\rho$ with respect to the right action $\curvearrowleft$,
    \item Compatibility of the right action $\curvearrowleft$ with the product of $G$.
\end{enumerate}


\subsection{Outline}
\label{ssec:outlein}

$\sbt$ After the introduction, the article is divided into two parts. In the first part (Sections \ref{sec:operadsbrace} and \ref{sec:twistedmultiplicativity}) we provide the reader with the necessary background on operads and brace algebras. We then define the two main operads for the present work, which are the operad of non-crossing partitions (see Definition \ref{def:setpartitions}) and the operad of  non-crossing linked partitions (see Definition \ref{def:operadnoncrossinglinkedpartition}). In Subsection \ref{ssec:operadwithmult}, we define operads with multiplication as well as brace algebras.

$\sbt$ In Section \ref{sec:twistedmultiplicativity}, we first address the problem of computing operator-valued free cumulants with products as entries and give a fixed point equation for computing the free cumulants of the product of two operator-valued free random variables, see Subsection \ref{ssec:operadwithmult}. To the extend of our knowledge, this fixed point equation is new.

$\sbt$ Given this fixed point equation, we deduce a short proof of Theorem \ref{thm:twistedmultplicativityun}, pointing out the relation in Proposition \ref{prop:distributivity} as the key property of the concatenation and composition products for the twisted factorization to hold.

$\sbt$ In Subsection \ref{ssec:freeproductoperadmultiplication}, we define the free product and  the $T$-transform in this abstract setting.


\subsection{Basic notions and notations}

We recall the definition of an operator-valued non-commutative probability space  \cite{mingo2017free}.

\begin{definition}[Operator-valued probability space]\label{Def:OVPS}
An operator-valued probability space is a triple $(\mathcal{A},\phi,B)$ such that
\begin{enumerate}[\indent 1.]
	\item $\mathcal{A}$ is a von-Neumann algebra\footnote{$\mathcal{A}$ can also be $C^{\star}$ algebra, depending on the type of functional calculus on $\mathcal{A}$ we want to have access to, either borelian, continuous or polynomial.},
	\item
	$B$ is a $C^{\star}$ algebra,
	\item
	$\mathcal{A}$ is a $B$-$B$ bimodule, meaning that $B$ acts from the left as well as from right on $\mathcal{A}$
	\begin{equation*}
		b_{1}\cdot(a\cdot b_{2}) = (b_{1} \cdot a)\cdot b_{2},\quad (b \cdot a)^{\star} = a^{\star}\cdot b^{\star}
		\qquad b,b_{1},b_{2} \in B, a \in \mathcal{A}.
	\end{equation*}
	\item
	The state, $\phi : \mathcal{A} \to B$, is a $B$-$B$ bimodule morphism\footnote{However, it is \emph{not} an algebra morphism.} which is positive:
	\begin{equation*}
		\phi(aa^{\star}) \in BB^{\star}
		\qquad a\in\mathcal{A}.
	\end{equation*}
\end{enumerate}
\end{definition}

In the present work, functionals of random variables are restricted to polynomial ones. As a consequence, we can drop all topological assumptions in the above definition. In particular, $\mathcal{A}$ and $B$ are only assumed to be involutive algebras.


\section{Operads and brace algebras}
\label{sec:operadsbrace}


In this section, we recall the definitions of operads and brace algebras. The reader is directed to the monograph \cite{loday2012algebraic} for a detailed introduction.

\subsection{Algebraic planar operads}
\label{ssec:operads}

 Operads are models for composing operators with multiple inputs and a single output. They have been introduced by Boardman and Vogt in the 1970's. There exist multiple equivalent definitions of an operad. In this section, we adopt a rather algebraic point of view by defining an operad first as a monoid in a certain monoidal category that we introduce. It should be understood that what we call an operad in this article is also called a planar operad; the set of operators we consider are vector spaces and are not assumed to be endowed with an action of the symmetric group.

 Operators are organized in a \emph{collection} $C$, that is a sequence of vector spaces $(C(n))_{n\geq 1}$. The vector space $C(n)$ comprises all operators with $n$ inputs. We assume this collection to be reduced which means that all operators have at least one input. The number of inputs of an operator $p$  is denoted $|p|$ and is called its \emph{arity}.

A morphism between two collections $C$ and $D$ is a sequence of linear morphisms $(\phi(n))_{n\geq1}$ with $\phi({n}): C(n) \rightarrow D(n)$,~$n \geq 1$. We denote by $\mathrm{Coll}$ the category of all collections. We remark that it is an abelian category in an obvious way, the sum $C\oplus D$ of two collections is the collection defined by
\begin{equation}
    (C\oplus D)(n)=C(n)\oplus D(n).
\end{equation}

The category $\mathrm{Coll}$  can be endowed with a monoidal structure, where the tensor product, here denoted $\sbt$, is the $2$-functor from $\mathrm{Coll} \times \mathrm{Coll}$ to $\mathrm{Coll}$ defined by:
\begin{align*}
	 & (C \sbt D)(n) = \bigoplus_{\displaystyle{\substack{k \geq 1 \\n_{1}+\cdots+n_{k}=n } }}
	C(k) \otimes D(n_{1}) \otimes \cdots \otimes D(n_{k}),         \\
	 & (f \sbt g)(n)= \bigoplus_{\displaystyle{\substack{k \geq 1  \\n_{1}+\cdots+n_{k}=n } }}
	f(k) \otimes g(n_{1}) \otimes \cdots \otimes g(n_{k}).
\end{align*}
The unit element for this product on $\mathrm{Coll}$ is the collection denoted by $\mathbb{C}_{\sbt}$ such that $\mathbb{C}_{\sbt}(n)=\delta_{n=1}\mathbb{C}$, this means
\begin{equation}
    C \sbt \mathbb{C}_{\sbt} \simeq C ~{\rm and }~ \mathbb{C}_{\sbt} \sbt C \simeq C.
\end{equation}
Given two collections $C$ and $D$, the set ${\rm Hom}_{\rm Coll}(C,D)$ of collection homomorphisms from $C$ to $D$ is a vector space. However, the tensor product of two collection morphisms is linear only on its left argument, in particular $f\sbt\lambda g \neq \lambda f\sbt g$, for any $f,g\in {\rm Hom}_{\rm Coll}(C,D)$.

\begin{definition}[Operad]
An \emph{operad} $\mathcal{C}$ is a monoid in the monoidal category $(\mathrm{Coll},\sbt, \mathbb{C})$, i.e., a triple $(C, \gamma_{\scriptscriptstyle{\mathcal{C}}},\eta_{\scriptscriptstyle{\mathcal{C}}})$, $C \in \mathrm{Coll}$, with
$$
	\gamma_{\scriptscriptstyle{\mathcal{C}}}: C \sbt C  \rightarrow  C,~\eta_{{\scriptscriptstyle{\mathcal{C}}}}: \mathbb{C} \rightarrow {C},
$$
satisfying associativity and unitality constraints, namely
$$
\gamma_{\scriptscriptstyle{\mathcal{C}}} \circ (\gamma_{\scriptscriptstyle{\mathcal{C}}} \sbt \mathrm{id}_{{\scriptscriptstyle{\mathcal{C}}}})=
\gamma_{\scriptscriptstyle{\mathcal{C}}} \circ (\mathrm{id}_{{\scriptscriptstyle{\mathcal{C}}}} \sbt \gamma_{\scriptscriptstyle{\mathcal{C}}})
$$
$$
\gamma_{\scriptscriptstyle{\mathcal{C}}} \circ (\eta_{{\scriptscriptstyle{\mathcal{C}}}} \sbt \mathrm{id}_{{\scriptscriptstyle{\mathcal{C}}}})= \gamma_{\scriptscriptstyle{\mathcal{C}}}  \circ (\mathrm{id}_{{\scriptscriptstyle{\mathcal{C}}}} \sbt \eta_{{\scriptscriptstyle{\mathcal{C}}}}) = \mathrm{id}_{{\scriptscriptstyle{\mathcal{C}}}}.
$$
\end{definition}

It is common to use the symbol $\circ$ to denote the composition $\gamma_{{\scriptscriptstyle{\mathcal{C}}}}$,	$\gamma_{\scriptscriptstyle{\mathcal{C}}}\big(p \otimes q_{1} \otimes \cdots \otimes q_{|p|}\big) = p \circ (q_{1}\otimes \cdots \otimes q_{|p|}).$ In the remaining part of this article we follow this convention.
More practically, one can exploit the unital and associativity constraints to associate \emph{partial compositions} to any operad $(C, \gamma_{\scriptscriptstyle{\mathcal{C}}},\eta_{\scriptscriptstyle{\mathcal{C}}})$, which compose an operator $p$ in $C$ with another operator $q$ at a certain input of $p$:
\begin{equation*}
	p\circ_{k}q = p \circ (\textrm{id}_{\scriptscriptstyle{\mathcal{C}}}^{\otimes k-1} \otimes q \otimes \textrm{id}_{\scriptscriptstyle{\mathcal{C}}}^{\otimes |p|-k})
	\qquad
	1 \leq k \leq |p|.
\end{equation*}
Associativity of $\gamma_{\scriptscriptstyle{\mathcal{C}}}$ translates as follows for the partial compositions:
\begin{align*}
    (p \circ_{i} q) \circ_{j} r
    &= (p \circ_{j} r) \circ_{i+|r|-1} q \qquad 1 \le j < i \le |p|\\
   p\circ_{i} (q \circ_{j} r)
   & = (p\circ_{i} q) \circ_{j+i-1} r
   \qquad 1 \le i \le |p|,\ 1\le j \le |q|.
\end{align*}
In the next sections, we define two operads on non-crossing set partitions and on non-crossing linked set partitions. These two operads are in fact set-operads that we regard as operads (in the category of vector spaces) with preferred basis that behave well under the operadic products. For the time being, let us give two examples of operads.
\begin{example}
\label{ex:operads}
	\begin{enumerate}[1.]
		\item The forgetful functor from the category of operads to the category of collections admits a left adjoint that we call the free functor $\mathcal{F}$. Given a collection $C$, the free operad $\mathcal{F}(C)$ over $C$, is spanned by \emph{planar rooted trees} with internal vertices decorated with elements in the collection $C$.  Recall that an internal vertex of a tree is a vertex with at least one input. The degree of a decoration of an internal vertex matches the number of its inputs.
		In a linear setting, that is if $C$ is a collection of vector spaces, we identify a tree $\tau$ having a vertex decorated by a sum $a+b$ of elements $a,b \in C$ with the sum of two trees, each obtained by replacing the decoration $a+b$ with $a$ or $b$.
        A leaf is a vertex of a tree with no inputs. The number of leaves of a tree is the number of its inputs. The collection $C$ corresponds to decorated corollas, trees with only one internal vertex. Composition of a tree $\tau$ with another tree $\tau^{\prime}$ at the $i^{th}$ input of $\tau$ (the leaves of $\tau$ are ordered from the leftmost to the rightmost leaf)  is obtained by \emph{grafting} the root of $\tau^{\prime}$ to the leaf of $\tau^{\prime}$. The identity is the tree with a single vertex.

		\item Given a vector space $V$, we denote by $\mathrm{End}_{V}$ the operad whose underlying collection consists of all multilinear maps on $V$,
		      \begin{equation*}
			      \mathrm{End}_{V}(n) = \mathrm{Hom}_{\mathrm{Vect}}(V^{\otimes n}, V),
		      \end{equation*}
		and the composition is induced by composition of functions.

        \item \label{ex:wordinsertionoperad} \emph{Operad $\mathcal{W}$} of word-insertions. Let $\mathcal{A}$ be an alphabet. Set $T(\mathcal{A}) := \bigoplus_{n\geq 1} \mathcal{A}^{n}$, the space of all non-commutative polynomials in elements from $\mathcal{A}$, and $\bar{T}(\mathcal{A}) := \mathbb{C}\cdot\emptyset \oplus T(\mathcal{A})$. Then $\bar{T}(\mathcal{A})$ is an unital algebra for the concatenation product, denoted $\cdot$, and with unit $\emptyset$.
		The degree $|w|$ of an element $w=w_1 \cdots w_n \in \mathcal{A}^{\otimes n}$ is $|w|=n+1$ and $|\emptyset|=1$. We turn $\bar{T}(\mathcal{A})$ into the operad $\mathcal{W}$ by defining the operadic composition $\gamma_{\scriptscriptstyle{\mathcal{W}}}$ for a word $w=w_1 \cdots w_n$:
		\begin{equation*}
			\gamma_{\scriptscriptstyle{\mathcal{W}}}(w \otimes u_{0}\otimes\cdots\otimes u_{n}) = u_{0}\cdot w_{1} \cdot u_{1} \cdots u_{n-1}\cdot w_{n}\cdot u_{n}.
		\end{equation*}
	\end{enumerate}
\end{example}

The definition of a planar operad uses the monoidal structure of the category Vect$_{\mathbb{C}}$ of all vector spaces. Replacing the monoidal category (Vect$_{\mathbb{C}}$) by another yields the notion of an operad in a symmetric monoidal category. For example, one can replace Vect$_{\mathbb{C}}$ by the category Set of sets with bijections. It is a monoidal category for the cartesian product of sets. A monoid in the category of all set collections is called a set operad.

We recall here two different tensor products on the category of operads, that will be  briefly used in the forthcoming sections.

\begin{definition}[Hadamard product]
	\label{def:hadamardproduct}
	Let $\mathcal{P}=(P,\gamma_{\scriptscriptstyle{\mathcal{P}}})$ and $\mathcal{Q}=(Q,\gamma_{\scriptscriptstyle{\mathcal{Q}}})$ be two operads. The Hadamard product of $\mathcal{P}$ and $\mathcal{Q}$ is the operad $\mathcal{P}\otimes_{H}\mathcal{Q} = (P\otimes_{H} Q,\gamma_{\scriptscriptstyle{\mathcal{P}\otimes_{H}\mathcal{Q}}})$ defined by
	\begin{align*}
		&\left(P\otimes_{H}Q\right)(n) = P(n)\otimes Q(n) \\
		&\gamma_{\scriptscriptstyle{\mathcal{P}\otimes_{H}\mathcal{Q}}}(p\otimes q\,\otimes\, (p_{1}\otimes q_{1} \otimes\cdots\otimes p_{n}\otimes q_{n})) = \gamma_{\scriptscriptstyle{\mathcal{P}}}(p\otimes p_{1}\otimes\cdots\otimes p_{n}) \otimes \gamma_{\scriptscriptstyle{\mathcal{Q}}}(q \otimes q_{1}\otimes\cdots\otimes q_{n}).
	\end{align*}
\end{definition}

\begin{definition}[Free product]
	\label{def:freeproduct}
	Let $\mathcal{P}=(P,\gamma_{\scriptscriptstyle{\mathcal{P}}})$ and $\mathcal{Q}=(Q,\gamma_{\scriptscriptstyle{\mathcal{Q}}})$ be two operads. The \emph{free product} of $\mathcal{P}$ and $\mathcal{Q}$ is the operad $\mathcal{P}\sqcup \mathcal{Q}$ obtained by quotiening the free operad on the collection $\mathcal{P}\oplus\mathcal{Q}$ by relations in the operad $\mathcal{P}$ as well as relations in the operad $\mathcal{Q}$, but no other.
\end{definition}

It may happen that operators which we want to compose have different input and output ranges. A model for composing such operators is called a \emph{coloured operad}. Collections are replaced by coloured collections, each vector space $C(n)$ of operators with $n$ inputs is split into a direct sum of spaces $C^{c^{\prime}}_{c_{1},\ldots,c_{n}}$, $c_{1},\ldots,c_{n},c^{\prime} \in {\sf C}$ comprising all operators with source spaces labeled $c_{1},\ldots,c_{n}$ and target labeled $c^{\prime}$ for some set of colors ${\sf C}$. Formal composition of collections (the monoidal product $\bullet$) admits a coloured version, for which operators are formally composed provided that colorations of inputs and outputs match.

\subsection{Operads of non-crossing partitions}
\label{ssec:ncOperadlinkedPart}
\subsubsection{The gap insertion operad}
\label{ssec:partitionoperad}

In this section we recall the definition of the \emph{gap insertion operad} on non-crossing partitions first introduced in \cite{ebrahimi2019operads}. We define a new operad on \emph{non-crossing linked partitions}. Note that we adopt a more general definition of the latter, which are in fact coverings (blocks may intersect).

To fix notations, we denote by $\nc(n)$ the set of all non-crossing partitions of the interval $\llbracket 1,\ldots,n \rrbracket$ (with its natural order). Recall that $\pi \in \nc(n)$ if it is a partition of $\llbracket 1,\ldots,n \rrbracket$ and no blocks of $\pi$ cross, which means that for any sequence of integers $a < b < c < d$ in $\llbracket 1,\ldots,n \rrbracket$, one has for two blocks $V,W \in \pi$ that
\begin{equation}
    {~\rm if~} a,c \in V\ \text{and}\ ~ b,d\in W,~ {~\rm then ~}V=W.
\end{equation}

The operadic view on a partition $\pi \in \nc(n)$ is that of an operator with $n+1$ inputs ($|\pi|=n+1$). An input corresponds to a gap between two consecutive elements of the partitioned set, including the front gap before the element $1$ and the back gap after the element $n$. Hence, we may therefore insert $n+1$ non-crossing partitions into a non-crossing partition by stuffing them into the gaps of the latter.

\begin{definition}[Operad $\mathcal{N}\mathcal{C}$ of non-crossing partitions]
	\label{def:setpartitions}
	We set $\nco(n) := \mathbb{C}\left[\nc(n-1)\right]$. In particular, we have $\nco(1)=\mathbb{C}\{\{\emptyset\}\}$. The unique partition of the empty set acts as the operad unit. Let $\pi$ be a non-crossing partition and $(\alpha_{1},\ldots,\alpha_{|\pi|})$ a sequence of non-crossing partitions, we define
	\begin{equation*}
		\gamma_{\scriptscriptstyle{\nco}}(\pi \otimes \alpha_{1} \otimes \cdots \otimes \alpha_{|\pi|})
		= \bigcup_{i=1}^{|\pi|}\{i-1+b,~b \in \alpha_{i}\} \cup \tilde{\pi},
	\end{equation*}
	where $\tilde{\pi}$ is the non-crossing partition of $\{|\pi_{1}|, |\alpha_{1}|+|\alpha_{2}|,\ldots,|\alpha_{1}|+\cdots+|\alpha_{|\pi|}|\}$ induced by $\pi$.
\end{definition}

\begin{figure}[!ht]
	\centering
	\includegraphics[scale=0.85]{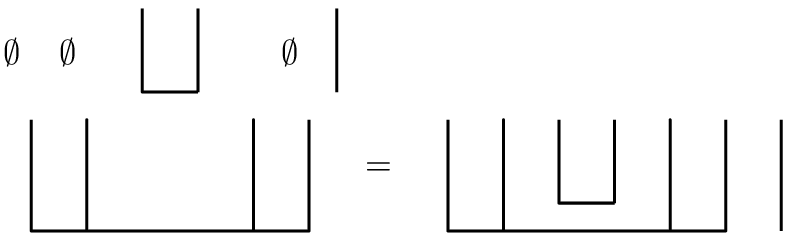}
	\caption{\label{fig:operadinsertion} Example of composition in the gap-insertion operad $\mathcal{N}\mathcal{C}$.}
\end{figure}

\begin{definition}[Coloured non-crossing partitions]
	A coloured non-crossing partition is an element $\pi \otimes w$, $|w|= |\pi|$ in the Hadamard tensor product $\mathcal{N}\mathcal{C} \otimes_{H} \mathcal{W}$ (see item \ref{ex:wordinsertionoperad} in Example \ref{ex:operads} for the definition of the word-insertions operad $\mathcal{W}$).
\end{definition}

The gap-insertion operad of non-crossing partitions admits the following presentation in terms of generators and relations.

\begin{lemma}[Prop.~3.1.4 in \cite{ebrahimi2019operads}]
	For any $n\geq 1$, we put $\mathbb{1}_{n+1} = \{\llbracket 1,n\rrbracket\}$. Then the operad $(\mathcal{N}\mathcal{C}, \gamma_{\scriptscriptstyle{\mathcal{N}\mathcal{C}}})$ is generated by the elements $\mathbb{1}_{n}$, $n \geq 1$ with the relation:
	\begin{equation*}
		\forall m,n \geq 1,\quad \mathbb{1}_{m} \circ_{m} {\mathbb{1}}_{n} = \mathbb{1}_{n} \circ_{1} \mathbb{1}_{m}.
	\end{equation*}
\end{lemma}


\medskip

Recall from \cite{gilliers2020shuffle} that the distribution of a random variable $a \in \mathcal{A}$ in an operator-valued probability space $(\mathcal{A}, \phi, B)$ yields an operadic morphism ${\hat\Phi}_{a}$ on the gap-insertion operad with values in the operad of endomorphisms  $\mathrm{End}_{B}$ of $B$, prescribed by
\begin{equation}
\label{eqn:operadicmorphismmoments}
	{\hat\Phi}_{a}(\mathbb{1}_{n})(b_{0},\ldots,b_{n}) := \phi(b_{0}ab_{1}a\cdots ab_{n})
\end{equation}
with $b_{0},\ldots,b_{n} \in B$. Because $\phi$ is $B$-$B$ bimodule map the morphism ${\hat\Phi}_{a}$ is well-defined.

The free cumulants $\{\kappa_n(a)\}_{n>0}$,  $a \in \mathcal{A}$, correspond to another operadic morphism
$$
    {\sf \hat{K}}_{a}:\mathcal{N}\mathcal{C} \rightarrow \mathrm{End}_{B}
$$
such that
\begin{equation}
\label{eqn:operadicmorphismcumulants}
	{\sf \hat{K}}_{a}(\mathbb{1}_{n})(b_{0},\ldots,b_{n}) = \kappa_{n}(b_{0}ab_{1},\ldots,a b_{n}).
\end{equation}


\subsubsection{Nesting-or-linking operad}
\label{sssec:noncrossinglinkedpartition}


We now give the definition of a so-called non-crossing linked partition. As said, our definition is more general than what is usually given in the literature.

\begin{definition}[Non-crossing linked partitions] Let $n$ be a positive integer. A \emph{non-crossing linked {\rm{(ncl)}} partition} is a collection $\pi$ of subsets (blocks) of $\llbracket 1,n\rrbracket$ such that:
	\begin{enumerate}[$\indent 1.$]
		\item $\bigcup_{V \in \pi} V = \llbracket 1,n\rrbracket$
		\item for $U$ and $V$ two blocks of $\pi$, if $ a < c < b < d$ with $a,b \in U$ and $c,d \in V$, then $V = U$,
		\item if $U\neq V$ and $|U|, |V| > 1,~|U \cap V| > 0$, then $U \cap V = \{x\}$ and $x = \min U$ or $x = \min V$.
	\end{enumerate}
	For $n \geq 1$, we denote by $\ncl(n)$ the set of non-crossing linked partitions of $\llbracket 1,n\rrbracket$.
\end{definition}
We refer the reader to Figure~\ref{fig:nclpartition} for examples of non-crossing linked partitions.

\begin{figure}[!ht]
	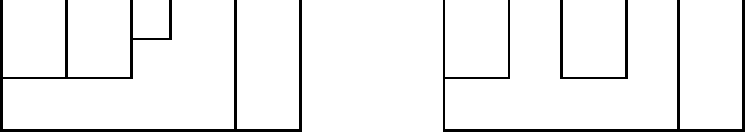
	\caption{\label{fig:nclpartition} On the left, we have a connected ncl partition. On the right, we have a ncl partition.}
\end{figure}

Two blocks of a ncl partition are allowed to meet at their minimal elements, this is the main difference between our definition of ncl partitions and the one given in \cite{dykema2007multilinear}. It allows us to define a natural operadic structure on ncl partitions. We denote $\ncl_{D}$ the subset of all non-crossing linked partitions with no pairs of blocks intersecting at their minimal elements.

\begin{definition}[Nesting-or-linking operad]
	\label{def:operadnoncrossinglinkedpartition}
	Define the degree $|\pi|$ of a ncl partition by $|\pi| = n$ if $\pi \in \ncl(n)$. We denote by $\mathcal{NCL}(n)$ the vector span of $\ncl(n)$, ${n \geq 1}$. The operadic composition $\gamma_{\scriptscriptstyle{\mathcal{NCL}}}:\mathcal{NCL}\sbt \mathcal{NCL}\to \mathcal{NCL}$
	is defined as follows. Given a ncl partition $\alpha$ in $\ncl(n)$ and $\beta_{1},\ldots,\beta_{n}$ a sequence of ncl partitions, we define:
	\begin{equation*}
		\gamma_{\scriptscriptstyle{\mathcal{NCL}}}(\alpha \otimes \beta_{1} \otimes\cdots\otimes\beta_{n})
		=\bigcup_{i=1, \ldots, n} \{ |\beta_{i-1}| + V ,~V \in \beta_{i} \} \cup \tilde{\alpha},
	\end{equation*}
	where $\tilde{\alpha}$ is the ncl partition of the set $\{1, 1+|\beta_{1}|, |\beta_{1}|+|\beta_{2}|+1,\ldots, |\beta_{1}| + \cdots + |\beta_{n-1}| + 1 \}$ induced by $\alpha$.
	We denote by $\pmb{|}$ the unique element in $\ncl(1)$, which plays the role of the unit for $\gamma_{\scriptscriptstyle{\mathcal{NCL}}}$.
\end{definition}

\begin{figure}
    \centering
    \includegraphics{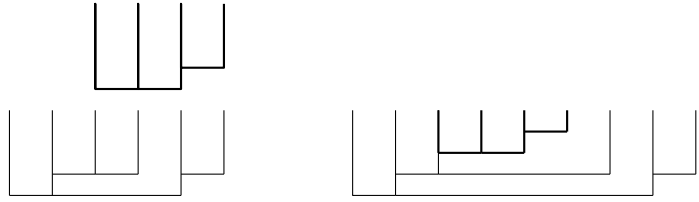}
    \caption{Example of a composition in the operad $\mathcal{NCL}$. We insert the ncl partition drawn in bold into the one drawn below it at the third input of the latter.}
    \label{fig:excomposition}
\end{figure}

To help understanding composition of ncl partitions in $\mathcal{NCL}$, we introduce orders on the blocks of a ncl partition. Pick a ncl partition $\pi$. Two blocks of $\pi$ can be \emph{nested} or \emph{linked}. The former refers to the case where a block $V$ of $\pi$ is contained disjointly in the convex hull, ${\sf Conv}(W)$, of another block $W$. We write $W \leftarrow V$ in that case,
\begin{equation*}
W \leftarrow V \Leftrightarrow V \subset {\sf Conv}(W),~V\cap W = \emptyset.
\end{equation*}

The transitive closure of this elementary relation on $\pi$ yields a partial order on $\pi$ which is denoted $\leftarrow$.

Linking refers to the case for which the minimum of a block $V$ is contained in another block $W$. We distinguish two cases,  $\min(V)=\min(W)$ and $V$ is contained in the convex hull of $W$ or $\min(V)\neq \min(W)$. We write $W \leftsquigarrow  V$ if $V$ is linked to $W$,
\begin{equation*}
W \leftsquigarrow V \Leftrightarrow (\min(V)=\min(W),~V \subset {\sf Conv}(W)) ~\mathrm{or}~ (\min(V)\in W,~\min(V)\neq\min(W)).
\end{equation*}

Again, taking the transitive closure of this elementary relation yields an order on $\pi$, that we denote $\leftsquigarrow$.


\begin{definitionproposition}[Nesting or Linking order $\Leftarrow$]
   Pick $\pi$ a non-crossing linked partition. The Nesting-or-linking order $\Leftarrow$ on $\pi$ is given by
   \begin{equation}
       V \Leftarrow W \Leftrightarrow V \leftarrow W {~\rm or~} V \leftsquigarrow W.
   \end{equation}
\end{definitionproposition}

\begin{proof}
It is clear that $\Leftarrow$ is transitive and reflexive. Pick two blocks $V$ and $W$ with
$V\Leftarrow W$ and $W \Leftarrow V$. The non-trivial cases are the following ones
\begin{enumerate}
    \item $V \leftarrow W$ and $W \leftsquigarrow V$,
    \item $V \leftsquigarrow W$ and $W \leftarrow V$.
\end{enumerate}
In the first case, from $V \leftarrow W$, the convex hull of $V$ contains the convex hull of $W$ and the two blocks are disjoint or equal. Now, $W\leftsquigarrow V$ implies either ${\sf Conv}(V) \subset {\sf Conv}(W)$ or $\min(V)=\max(W)$ or ${\sf Conv}(V)\cap {\sf Conv}(W)=\emptyset.$ The latter alternative can not hold since ${\sf Conv}(V) \subset {\sf Conv}(W)$. If $\min(V)=\max(W)$ the two blocks are not disjoint and are therefore equal. If ${\sf Conv}(W)\subset {\sf Conv}(V)$ then $V=W$ because ${\sf Conv}(V)\subset {\sf Conv}(W)$.
\end{proof}
\begin{figure}[!ht]
\centering
\includesvg[scale=0.9]{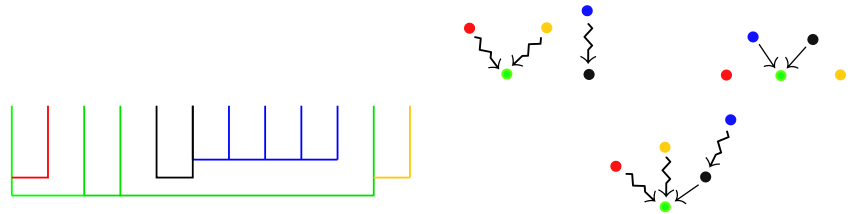}
\caption{\label{fig:hassediagram} Example of a ncl partition together with the Hasse diagrams for the order $\leftsquigarrow$ (on the top left), the order $\leftarrow$ (on the top right) and the order $\Leftarrow$ (at the bottom).}
\end{figure}

\begin{lemma}
Let $\pi$ be a ncl partition. The Hasse diagram of the order $\Leftarrow$ on $\pi$ is a forest.
\end{lemma}

Assume that $\tau_{\Leftarrow}(\pi)$ contains a cycle and pick one cycle $c$ with minimal length.
If the cycle $c$ has length two, there exists two blocks $V$ and $W$, $V\neq W$, of $\pi$ with $V\Leftarrow W$ and $W\Leftarrow V$. One can assume that $V \leftsquigarrow W$ and $W \leftarrow V$. In that case, $V$ is included in the convex hull of $W$, $V\cap W=\emptyset$. At the same time $V\leftsquigarrow W$ implies that $\min(W) \in V$. Both can not hold and $c$ has a length greater than $2$. Since a Hasse diagram can not contain a cycle of length three, $c$ has length greater than three.

Assume that $c$ is not oriented. The cycle $c$ contains a triple of distinct blocks $U, V, W$ such that $U \Leftarrow V\Rightarrow W$. Without loss of generality, we can suppose that $U\leftsquigarrow V$ and $V \rightarrow W$. In particular, $V$ is included in the convex hull of $W$, $V\cap W = \emptyset$ and $\min(V) \in U$. Since $\pi$ is non-crossing, this implies that $U$ is included in the convex hull of $W$. Either $\min(U) \in W$ and $W \leftsquigarrow U$, either $\min(U)\not\in W$ and $W \leftarrow U$. In all cases, $W \Leftarrow U$. This entails that $\tau_{\Leftarrow}(\pi)$ contains a cycle of length $3$, which is not possible.

Assume that $c$ is oriented, $c=(c_1,\ldots,c_n)$, $c_i$ a block of $\pi$, $c_i\neq c_j$, $1 \leq i,j\leq n$, $c_i \Rightarrow c_j$, $c_n \Rightarrow c_1$. This is a general fact that a Hasse diagram of a poset has no oriented cycles, but let us recall the argument for the Hasse diagram of the poset of non-crossing linked partitions. Notice that for all $i,j \in [n]$, it holds that ${\sf Conv}(c_{i_j}) \cap {\sf Conv}(c_{i_k})\neq \emptyset$. Either ${\sf Conv}(c_i) \subset {\sf Conv}(c_{i+1})$ for $1 \leq i\leq n$ and in that case $c_1=\cdots c_n$ in contradiction with our hypothesis, either $c_{i_0} \rightsquigarrow c_{i_0+1}$ and $\min(c_{i_0})=\max(c_{i_0+1})$. Choose $i_0$ the smallest integer in $[n]$ satisfying this property. Hence, $c_{i_0-1}\Rightarrow c_{i_0}$ but also ${\sf Conv}(c_{i_0-1}) \cap {\sf Conv}(c_{i_0+1}) = \emptyset$. This leads to a contradiction.

As it will be clear after the proofs of the two propositions below, the Hasse diagram of $\Leftarrow$ is reminiscent of a tree monomial representing a ncl partition in the operad $\mathcal{NCL}$.

In \cite{dykema2007multilinear}, Dykema introduced two "projections" from the subset of non-crossing linked partitions $\ncl_D$ to non-crossing partitions to define a partial order on $\ncl_{D}$. We define one of these projections in our setting.

\begin{definition}[Block merging]
Given a ncl partition $\pi$, we denote by $\hat{\pi}$ the non-crossing partition obtained by merging all blocks of $\pi$ with a non-empty intersection.
\end{definition}

\begin{example}
    For example, if $\pi = \{ \{1,5\},\{1,3\},\{2,4\}\}$, $\hat{\pi} = \{ \{1,3,5\},\{2,4\}\}$.
\end{example}
Pick a ncl partition $\pi$. A \emph{connected component} of $\pi$ is a subset of blocks that once merged together form a block of $\hat{\pi}$.

\begin{proposition}[Connected non-crossing linked partitions]
\label{prop:ncconnected}
	Let $n\geq 1$ be an integer and define the following subset of $\ncl(n)$ of connected ncl partitions:
	\begin{equation*}
		\ncl^{(1)}(n) = \{\pi \in \ncl(n)\ |\ \hat{\pi} = 1_{n} \}.
	\end{equation*}
	Denote by $\mathcal{NCL}^{(1)}$ the collection of connected ncl partitions.
	Then, the operadic composition $\gamma_{\scriptscriptstyle{\mathcal{NCL}}}$ restricts to $\mathcal{NCL}^{(1)}$. In addition, $(\mathcal{NCL}^{(1)},\gamma_{\scriptscriptstyle{\mathcal{NCL}}} )$ is isomorphic to the free operad on the collection $I$ of single block non-crossing linked partitions.
\end{proposition}

\begin{proof}
Let $\alpha$ and $\beta$ be two connected ncl partitions and pick an integer $1\leq i \leq |\alpha|$. From the very definition of the operadic composition on $\mathcal{NCL}$ and as illustrated in Figure \ref{fig:excomposition}, the block of $\alpha$ containing $1$ intersects with the block of $\beta$ containing $i$ in the ncl partition $\beta\circ_{i}\alpha$. Hence, $\beta\circ_{i}\alpha$ is connected if $\alpha$ and $\beta$ are.


Then we construct a tree $\tau_{\leftsquigarrow}(\alpha)$ which is the Hasse diagram of $\leftsquigarrow$ augmented with leaves in order to interpret it as a tree monomial on (one block) ncl partitions. A vertex of $\tau_{\leftsquigarrow}(\alpha)$ corresponds to a block $V$ of $\alpha$ and has $|V|$ incoming edges. We connect the output of a vertex $V$ to an input of another block $W\in \pi$ if $V\leftsquigarrow W$ and if there is no other block $U$ such that $V \leftsquigarrow U \leftsquigarrow W$. Since $\pi$ is connected, there an unique minimal block for $\leftsquigarrow$ in $\pi$, the root of $\tau_{\leftsquigarrow}(\alpha)$.

Denote by $\mathcal{F}$ the free operad on the collection of one block ncl partitions and $p$ the canonical projection $p:\mathcal{F}\rightarrow \mathcal{NCL}$. Then $\tau_{\leftsquigarrow}(\alpha) \in \mathcal{F}$ and, clearly, $p(\tau_{\leftsquigarrow}(\alpha)) =  \alpha$.

We show next that any tree $\tau$ in $\mathcal{F}$ such that $p(\tau)=\alpha$ is equal to $\tau_{\rightarrow}(\alpha)$. We prove this fact by induction on the number of block of a ncl partition $\pi$. This is obvious if $\alpha$ has only one block.

Assume that the result holds for ncl partitions with at most $N$ blocks and pick a ncl partition $\alpha$ with $N+1$ blocks and $\tau \in \mathcal{F}$ such that $p(\tau)=\alpha$. Then $\tau$ has $N+1$ internal nodes. Write $\tau = V\circ(\tau_{1},\ldots,\tau_{|V|})$.
Notice that $V \leftsquigarrow W$ for any block $W$ in $p(\tau_i)$,  $1 \leq i \leq |V| $. Thus $V$ is the minimal block of $\alpha$ for the order $\leftsquigarrow$. We end the proof by applying the inductive hypothesis to the trees $\tau_{i}$ and ncl partitions $p(\tau_{i})$.
\end{proof}

Denote by I the collection of one block ncl partitions and by II the collection of ncl partitions defined by
\begin{equation}
    {\rm II}_n = \{\,\{\{1\}, \{2,\ldots,n-1\}\}\, \},~n\geq 2.
\end{equation}
We denote by $\theta_n$ the element of ${\rm II}_n$.

\begin{proposition}
    The operad $\mathcal{N}\mathcal{C}\mathcal{L}$ admits the following presentation
    \begin{equation}
\label{eqn:relations}
    \theta_n,\,1_n \quad \theta_n \circ_1 1_m = 1_m \circ_m \theta_n,\quad\theta_n \circ_1 \theta_m = \theta_m \circ_m \theta_n,~n,m\geq 2.
\end{equation}
\end{proposition}

\begin{proof}
    Denote by $\mathcal{F}$ the free operad on the collection ${\rm I}\cup {\rm II}$. The quotient of the free operad $\mathcal{F}$ by the relations \eqref{eqn:relations} is isomorphic to the collection $\tilde{\mathcal{F}}$ of trees with ${\rm I}\cup {\rm II}$ decorated internal vertices meeting the following constraint. If $v$ is an internal vertex of such a tree $\tau$ decorated with a ncl partition in the set II then its leftmost input is a leaf of $\tau$. We show then that any ncl partition can uniquely be written as a monomial in $\tilde{\mathcal{F}}$. The proof is done by induction on the number of blocks of a ncl partition.
    First, the result is trivial for one block ncl partitions. We assume the result to hold for any ncl partition with at most $N$ blocks and pick a ncl partition $\pi \in \ncl(p)$, $p\geq 2$ with $N+1$ blocks.
    Assume first there exists a tree $\tau \in \tilde{\mathcal{F}}$ such that $p(\tau)=\pi$ and write $\tau = V \circ (\tau_1,\tau_2,\ldots,\tau_{|V|})$, with $V$ a ncl partition in $I\sqcup II$. Follow two cases,
    \begin{enumerate}
        \item If $V \in {\rm II}$, then $\tau_1$ is the root tree and $\{1\} \in \pi$. In that case, $V$ is equal to $\theta_n$ where $n$ is the cardinal of the block $\{2<i_2 < \cdots <i_n-1\}$ containing $2$. If we let $\pi_j$ be the restriction of $\pi$ to the interval $\llbracket i_j+1,i_{j+1}-1\rrbracket$ with the convention that $\pi_j=\{1\}$ if $i_j +1 = i_{j+1}$ and $i_{n}=p$ then $p(\tau_j)=\pi_j$. The proof follows by applying the induction hypothesis to the ncl partitions $\pi_j$.
        \item If $V\in I$ then $V$ is equal to the block that contain $1$. The proof follows using the same line of arguments exposed in the previous case.
    \end{enumerate}
\end{proof}

To construct a tree monomial on ncl partitions in ${\rm I}\cup {\rm II}$ representing a ncl partition $\pi$, we start from the Hasse diagram $\tau_{\Leftarrow}(\pi)$ of $\pi$ for the order $\Leftarrow$.
First,  we augment $\tau_{\Leftarrow}$ with as many leaves as needed for the degree of each corolla in  $\tau_{\Leftarrow}$ to match the degree of the block of $\pi$ it is decorated with. We place the additional leaves so that if $W \in \pi$ meets $V$ at its $i^{th}$ element, the corolla $W$ is connected to the $i^{th}$ input of the corolla $V$. Do the same if $W$ is nested in $V$: if $W$ is contained between the $i^{th}$ element and the $(i+1)^{th}$ element, the corolla $W$ is connected to $i^{th}$ output of the corolla $V$. From this process results a forest of blocks of $\pi$. Then, follow this rules. Firstly,
\begin{enumerate}
    \item If $\{1\} \in \pi$, erase the corolla representing this block and decorate the corolla representing the block of $\pi$ containing $2$ by $\theta_n$ where $n$ is degree of this block minus one.
    \item If $\{1\}\not\in \pi$, decorate the corolla representing the block containing one with the corresponding block in ${\rm I}$.
\end{enumerate}
Secondly, if $V \leftsquigarrow W$ decorate the corolla representing $W$ with the corresponding element in I and if $V \leftarrow W$, decorate the corresponding corolla with $\theta_{|W|+1}$ and a leftmost leaf.
Finally, connect the root of each trees to the rightmost leaf of the previous one (if the forest is read from left to right).

\begin{corollaire}
The morphism of collections
\begin{equation}
    \begin{array}{cccc}
         j:& \mathbb{1} & \rightarrow & {\rm I} \\
         & \mathbb{1}_n &\mapsto &\theta_n
    \end{array}
\end{equation}
extends to an operadic morphism between the gap insertion operad and the linking-and-nesting operad.
\end{corollaire}
\begin{figure}[!ht]
\includesvg[scale=0.8]{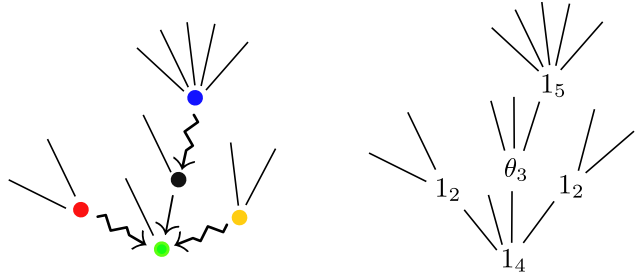}
\caption{\label{fig:monomial} Monomial representing the ncl partition drawn in Fig. \ref{fig:hassediagram}.}
\end{figure}
\medskip

Let us provide reason for introducing an operadic structure on non-crossing linked partitions. The $T$-transform of a random variable $a$ (in an operator-valued probability space $(\mathcal{A},\phi,B)$) with $\phi(a)=1_B$ is a sequence of multilinear maps on $B$,
\begin{equation}
    t_{a}(n): B^{\otimes n} \rightarrow B
    \qquad n\geq 0,
\end{equation}
with $t_a(0)=1_B$. This sequence can be inductively defined using the following formula relating the moments of $a$, understood as multilinear maps on $B$, to the sequence $(t_a(n))_{n\geq 0}$,
\begin{align*}
    \phi(b_{1}ab_{2}\cdots ab_{n})  &= \hspace{-0.5cm} \sum_{\pi \in \ncl_{D}(n)}\hspace{-0.2cm} t_{a}(\pi)(b_{1},\ldots,b_{n})\\
    &= t_{a}(n)(b_1,\ldots,b_n)+\hspace{-0.2cm}\sum_{\substack{\pi\in\ncl_{D}(n) \\ \pi\neq 1_n}}t_{a}(\pi)(b_{1},\ldots,b_{n}),
\end{align*}
with $b_1,\ldots,b_n \in B$. The function $\ncl_{D}(n) \ni \pi \mapsto t_a(\pi) \in {\rm End}_{B}$ is defined in \cite{dykema2007multilinear}, equations (67)-(70). Alternatively, by using the operadic structure we defined on non-crossing linked partitions, $t_{a}(\pi)(b_1,\ldots,b_{n})$ can be seen to match the value
$$
\hat{\sf T}_{a}(\pi)(b_1,\ldots,b_{n})
$$
of the operadic morphism $\hat{\sf T}_{a} : \mathcal{NCL} \rightarrow \mathrm{End}_{B}$ prescribed by the following equations
\begin{equation}
    \hat{\sf T}_{a}(1_{n})(b_{1},\ldots,b_{n-1}) = b_{1}\cdot t_{a}(n-1)(b_{2},\ldots,b_{n})
    \qquad n\geq 2.
\end{equation}

\subsection{Operads with multiplication}
\label{ssec:operadwithmult}


In this section we introduce the concept of multiplication in an operad, which is a distinguished operator of arity $2$. All operads we introduced so far admit such a multiplication.
We begin with the definition of \emph{brace algebra} and refer to \cite{aguiar2004infinitesimal}, \cite{ronco2001milnor}, and  \cite{chapoton2002theoreme} for details.
Recall that $T(A)$ denotes the vector space of all non-commutative polynomials with entries in $A$. We shall use word notation, $a_1\cdots a_n = a_1 \otimes \cdots \otimes a_n$, $a_i\in A$, for elements in $T(A)$.

\begin{definition}[Brace algebras] A brace algebra is a tuple $(A, \{-;-\})$, where $A$ is a vector space and $\{-;-\}$ is a linear map from $A \otimes T(A)$ to $A$ such that
	\begin{align}
		\label{eqn:bracealgeqn}
		\{ \{x; y_{1} \cdots y_{n}\} &; z_{1}\cdots z_{p}\} \\
		&= \sum_{I,J} \{x; z_{1} \cdots z_{i_{1}} \{y_{1}; z_{i_{1}+1} \cdots z_{j_{1}} \}z_{j_{1}+1}\cdots 
		\{y_{n}; z_{i_{n}+1} \cdots z_{j_{n}}\}z_{j_{n}+1} \cdots z_{p}\} \nonumber
	\end{align}
	with $x,y_{1},\ldots,y_{n}, z_{1},\ldots,z_{p} \in A$. The above sum runs over tuples $I=(i_1,\ldots,i_n)$ and $J=(j_1,\ldots,j_n)$ with
	$$0 \leq i_1 \le j_1 \le i_2 \le j_2 \le \cdots \le i_n \le j_n \le p.$$
\end{definition}
Let $\mathcal{P}=(P,\gamma_{\scriptscriptstyle{\mathcal{P}}})$ be an operad with $P(1)=\mathbb{C}\cdot {\rm id}$ and operadic composition $\gamma_{\scriptscriptstyle{\mathcal{P}}}$. It naturally yields a brace algebra structure on the collection $P$ (more precisely on the direct sum of the vector spaces $P(n),n\geq 1$),
\begin{align}
	\label{eqn:operadtogerstenhaber}
	\{x; y_{1}\cdots y_{n}\} &= \sum \gamma_{\scriptscriptstyle{\mathcal{P}}}(x\otimes\textrm{id}\otimes \cdots \otimes y_{1}\otimes\textrm{id} \otimes\cdots \otimes y_{2} \otimes\textrm{id}\otimes \\
	&\qquad\ \qquad\ \otimes y_{n-1}\otimes \textrm{id}\otimes \cdots\otimes \cdots \otimes y_{n} \otimes \textrm{id} \otimes \cdots \otimes\textrm{id}), \nonumber
\end{align}
where the sum runs over all possible ways to branch the operators $y_1, \ldots, y_n$ into $x$ while maintaining their linear order.
Equation \eqref{eqn:bracealgeqn} follows from associativity of $\gamma_{\scriptscriptstyle{\mathcal{P}}}$.

\begin{definition}[Multiplication in an operad]
An operator $m \in P(2)$ of arity $2$ satisfying
\begin{equation}
	\gamma_{\scriptscriptstyle{\mathcal{P}}}(m\otimes \operatorname{id}\otimes m)
	= \gamma_{\scriptscriptstyle{\mathcal{P}}}(m\otimes m \otimes   \operatorname{id}).
\end{equation}
is called a \emph{multiplication}.
\end{definition}
We now browse through some examples, found among the operads we introduced in the previous sections.
 \begin{example}
	 \begin{enumerate}[1.]
		 \item The operad $\nco$ of non-crossing partitions is an operad with multiplication $m= \pmb{|}$. In fact, $\gamma_{\scriptscriptstyle{\mathcal{NC}}}(\pmb{|} \otimes \{\emptyset\} \otimes \pmb{|})$ and $\gamma_{\scriptscriptstyle{\mathcal{NC}}}(\pmb{|} \otimes \pmb{|} \otimes \{\emptyset\})$ are both equal to the partition $\pmb{|\ |}$.
		 \item The operad $\mathrm{End}_B$ of multilinear maps on an algebra $\left(B, \mu_{B}\right)$ (with multiplication $\mu_B$), is an operad with multiplication $m=\mu_{B}$.
		 \item In the nesting-or-linking operad of non-crossing linked partitions we have two operators of arity two, the partitions $\pmb{|}~\pmb{|}=\{ \{1\},\{2\}\}$ and the partition $\pmb{\sqcup} = \{ \{1,2\} \}$. Only $\pmb{|}\,\pmb{|}$ is a multiplication, since $\pmb{|}\,\pmb{|} \circ_1 \pmb{|}\,\pmb{|} = \pmb{|}\,\pmb{|} \circ_2 \pmb{|}\,\pmb{|} = \pmb{|}\,\pmb{|}\,\pmb{|}$. For $\pmb{\sqcup}$, $\pmb{\sqcup}\circ_1\pmb{\sqcup} = \{ \{1,2\}, \{1,3\}\}$ and $\pmb{\sqcup}\circ_2\pmb{\sqcup} = \{\{1,2\},\{2,3\}\}$.
		 \item In the word-insertions operad $\mathcal{W}$ any letter $a \in \mathcal{A}$ provides an operator of arity two which is a multiplication, since $a\circ_1 a = aa = a\circ_2 a$.
	 \end{enumerate}
 \end{example}

\begin{remark}
In a graded context, that is, if the general term in the summation on the right hand side of \eqref{eqn:operadtogerstenhaber} is multiplied by a sign, existence of a multiplication in an operad provides a rich structure on the collection $P$ as observed by Gerstenhaber and Voronov in \cite{gerstenhaber1994homotopy}. In particular, the multiplication $m$ together with a certain \emph{graded pre-Lie product} yields a differential complex $(\mathcal{P},d)$.
\end{remark}

\medskip

We now assume that $\mathcal{P}$ is an operad with multiplication $m$. Following the above remark (that isno signs are involved), the non-graded pre-Lie product, denoted $\lhd$, is defined by:
\begin{equation*}
	x \lhd y = \{x;y\}
	\qquad x,y \in \mathcal{P}.
\end{equation*}
In fact, $x \lhd (y \lhd z) - (x \lhd y) \lhd z$ is symmetric in $y$ and $z$. We denote by $[\,-,\,-]$ the commutator bracket induced by the pre-Lie product $\lhd$:
\begin{equation*}
	[x,y] = x \lhd y - y \lhd x
	\qquad x,y \in \mathcal{P},
\end{equation*}
which satisfies the Jacobi identity.
\def\fsp{{\mathbb{C}[[\mathcal{P}]]}}
In \cite{chapoton2007relating}, the authors define a product, denoted $\times$, on the vector space $\fs$ of formal series on operators in $\mathcal{P}$,
\begin{equation}
    \fs = \displaystyle\prod_{n\geq 1} P(n),
\end{equation}
defined for two series $x,y \in \fs$ in terms of braces
\begin{equation*}
	x \times y  = \sum_{n \geq 1} \{ x_{n} ; y_{m_{1}}\cdots y_{m_{n}}\}.
\end{equation*}

\begin{proposition}[see Prop.~4.1 in \cite{chapoton2007relating}]
	$(\fs, \times, \mathrm{id})$ is an associative monoid. Moreover, $x \in \fs$ is an invertible element if and only if $x_{1} \neq 0$.
\end{proposition}

The above proposition implies that the subset $G \subset \fs$ defined by
\begin{equation*}
	G = \Big\{ x \in \fs : x_{1} = \mathrm{id} \Big\}
\end{equation*}
endowed with the composition $\times$ is a group. The multiplication $m$ on $\mathcal{P}$ yields another group product that we define in Subsection \ref{ssec:freeproductoperadmultiplication}. But first, the multiplication $m$ yields a bilinear non-unital associative product $\cdot$ defined by
\begin{equation}
	x \cdot y = \{m;xy\}
	\qquad x,y \in \mathbb{C}[[\mathcal{P}].
\end{equation}
\begin{proposition}
	$(\mathcal{P}, \{\,-;-\, \}, \cdot)$ is a Gerstenhaber--Voronov algebra, which means that
	\begin{equation}
		\label{eqn:distributivity}
		\{x\cdot y; z_{1} \cdots z_{p}\} = \sum_{k=0}^{p} \{x;z_{1}\cdots z_{k}\} \cdot \{y;z_{k+1}\cdots z_{p} \}
		\qquad x,y,z_{1},\ldots,z_{p} \in \fs.
	\end{equation}
\end{proposition}
Equation \eqref{eqn:distributivity} is key to the twisted factorization of the $T$-transform as explained below in Subsection \ref{ssec:freeproductoperadmultiplication}.


\section{Twisted factorization of the T-transform}
\label{sec:twistedmultiplicativity}

In this section, we give a concise graphical proof of Theorem 7.18 in reference \cite{dykema2007multilinear}. The starting point is a formula, for operator-valued free cumulants with the product of two free random variables as entries, expressed in the language of operads.  We then show how non-crossing linked partitions are naturally entering the picture due to degree reduction resulting from filling inputs of a multilinear map with the unit of the algebra $B$.


\subsection{Free cumulants of products of random variables}
\label{ssec:freecumulantproduct}


In this subsection, we explain how to compute the multilinear function series corresponding to free cumulants of the product of two free random variables as the solution of a certain fixed point equation. The proof of this fixed point equation \eqref{FEQ1} is sketched below. This formula is well known in the scalar-valued case. The authors have not been able to locate the operator-valued version in the literature.


Let us fix once and for all two free random variables $a$ and $b$ in the operator-valued probability space $(\mathcal{A}, \phi, B)$ with $\phi(a)=\phi(b)=1_B$, and recall that we denote by $K_{x}$, $x\in\{a,b,ab\}$, the multilinear function series
\begin{equation}
	{K}_{x} = 1 + \sum_{n\geq 1} K^{(n)}_x,\quad K^{(n)}_x(b_1,\ldots,b_n)=\kappa_n(xb_1,xb_2,\ldots,xb_n,x),~b_1,\ldots,b_n \in B.
\end{equation}
Recall that in the scalar-valued case, i.e., when $B = \mathbb{C}$, one has the intriguing formula \cite{nica2006lectures}
\begin{equation}
	\label{eqn:krewerascomplement}
	\kappa_{n}(ab,\ldots,ab) = \sum_{\pi \in \nc(n)} \kappa_{\pi}(a) \kappa_{{\rm Kr}(\pi)}(b).
\end{equation}
Here, the non-crossing partition ${\rm Kr}(\pi) \in \nc(n)$ is the Kreweras complement of $\pi \in \nc(n)$, first introduced in \cite{kreweras1972partitions}. For two non-crossing partitions $\alpha$ and $\beta$ in $\nc(n)$, one denotes by $\alpha\cup\beta$ the partition of the interval $\llbracket 1,2n \rrbracket$ whose restriction to the odd integers, respectively to the even integers, coincides with $\alpha$, respectively $\beta$. By definition, $\rm{Kr}(\pi)$ is the maximal non-crossing partition (for the refinement order) such that $\pi \cup \rm{Kr}(\pi)$ is a non-crossing partition of $\nc(2n)$.

In the operator-valued case, since the cumulants of $a$ and $b$ do not commute with each other (they are elements of the non-commutative algebra $B$), the right-hand side of equation \eqref{eqn:krewerascomplement} does not factorise over $\pi$ and its Kreweras complement $\rm{Kr}(\pi)$. In fact, we should maintain the linear order between random variables in the word $a\otimes b\otimes\cdots \otimes a\otimes b$.


\medskip

Recall that the operadic morphisms $\hat{\sf K}_a$ and $\hat{\sf K}_b$ from $\mathcal{N}\mathcal{C}$ to the operad of endomorphisms $\mathrm{End}_{B}$ of $B$ are defined in equation \eqref{eqn:operadicmorphismcumulants}. The free product $\hat{\sf K}_{a} \sqcup \hat{\sf K}_{b}$ is the unique operadic morphism on the operad $\mathcal{N}\mathcal{C} \sqcup \mathcal{N}\mathcal{C}$ such that $(\hat{\sf K}_{a}\sqcup \hat{\sf K}_{b})(\pi) = \hat{\sf K}_{a}(\pi)$ if all blocks of $\pi$ are coloured with
$0$ (that is $\pi$ belongs to the first copy of $\mathcal{N}\mathcal{C}$ in $\mathcal{N}\mathcal{C}\sqcup\mathcal{N}\mathcal{C}$) or
$(\hat{\sf K}_{a} \sqcup \hat{\sf K}_{b}) (\pi) = \hat{\sf K}_{b}(\pi)$
if all blocks of $\pi$ are coloured with $1$ ($\pi$ belongs to the second copy of $\mathcal{N}\mathcal{C}$ in $\mathcal{N}\mathcal{C}\sqcup\mathcal{N}\mathcal{C}$).

With these definitions, we \emph{claim} that the operator-valued counterpart of formula \eqref{eqn:krewerascomplement}, computing cumulants of the product of two free random variable, reads
\begin{equation}
	\label{eqn:krewerascomplementop}
	x_{0}\kappa_{n}(ay_{1}bx_{1},ay_{2}bx_{2},\ldots,ay_{n}b)x_{n} = \sum_{\pi \in \nc(n)} (\hat{\sf K}_{a} \sqcup \hat{\sf K}_{b})(\tilde{\pi}) (x_{0},y_{1},x_{1},\ldots,y_{n},x_{n}),
\end{equation}
with $x_{0}, \ldots, x_{n} \in B$ and $y_{1},\ldots,y_{n} \in B$.
Here, we will denote by $\tilde{\pi}$ the non-crossing partition $\pi \cup \rm{Kr}(\pi)$ of $\llbracket 1,2n \rrbracket$. Each block of $\tilde{\pi}$ is coloured with $0$ or $1$, according to the parity of the elements in the block yielding an element of the free product $\mathcal{NC} \sqcup \mathcal{N}\mathcal{C}$. Validity of formula \eqref{eqn:krewerascomplementop} will be ascertained by the derivation of the correct formula for the $T$-transform of the product $ab$ in terms of the $T$-transforms of $a$ and $b$. We therefore only briefly comment on how the proof in the scalar case $(B=\mathbb{C})$ can be adapted to show \eqref{eqn:krewerascomplementop} in the particular case $n=3$. The cumulant $x_0\kappa_{3}(ay_0bx_1, ay_1bx_2, ay_2bx_3)$ is a sum over the non-crossing partitions $\pi$ of $\llbracket 1,6\rrbracket$ with $\pi \vee \{\{1,2\},\{3,4\},\{5,6\}\}=1_3$ of the partitioned mixed cumulants of ($ay_0$,$bx_1$,\ldots,$ay_2$,$bx_3$). Because $a$ and $b$ are two free random variables, blocks of such a non-crossing partition contain one variable or the other, but not mixed ones. These leave us with only five non-crossing partitions, $\{\{1,3\},\{5\}\}\cup \{\{2\},\{4,6\}\}$, $\{\{1\},\{3,5\} \}\cup \{\{2,6\},\{4\}\}$, $\{ \{1,5\},\{3\}\}\cup\{ \{2,4\},\{6\}\} $ and
$\{ \{1,3,5\} \} \cup \{ \{2\},\{4\},\{6\}\}$,$\{ \{ \{1\},\{3\},\{5\} \cup \{2,4,6\} \}\}$. The previous non-crossing partitions are of the form $\pi \cup {\sf Kr}(\pi)$, but in contrast to the scalar case the partitioned cumulant $\kappa_{\pi \cup {\sf Kr}(\pi)}(ay_0,bx_1,ay_1,bx_2,ay_2,bx_3)$ is not the product of the cumulants $\kappa_{\pi}(ay_0,ay_1,ay_2)$ and $\kappa_{{\sf Kr}(\pi)}(bx_1,bx_2,bx_3)$ because blocks of $\pi$ and ${\sf Kr}(\pi)$ are nested one into the others. For the above listed partitions the corresponding cumulants are
\begin{align*}
	\{\{1,3\},\{5\}\}\cup \{\{2\},\{4,6\}\}    &&  \kappa_2(ay_0\kappa_1(bx_1),ay_1)\kappa_{2}(bx_2\kappa_{1}(ay_2),bx_3)\\
	\{\{1\},\{3,5\} \}\cup \{\{2,6\},\{4\}\}   &&  \kappa_1(ay_0)\kappa_2(bx_1\kappa_2(ay_1\kappa_1(bx_2),ay_2),bx_3)                                                   \\
	\{ \{1,5\},\{3\}\}\cup\{ \{2,4\},\{6\}\}   && \kappa_2(ay_0\kappa_2(bx_1\kappa_1(ay_1),bx_2),ay_2)\kappa_1(bx_3)                                                      \\
	\{ \{1,3,5\} \} \cup \{ \{2\},\{4\},\{6\}\}&& \kappa_3(ay_0\kappa_1(bx_1),ay_1\kappa_1(bx_2),ay_2\kappa_1(bx_3)) \\
	\{ \{2,4,6\} \} \cup \{ \{1\},\{3\},\{5\}\}&& \kappa_3(\kappa_1(ay_0)bx_1,\kappa_1(ay_1)bx_2,\kappa_1(ay_2)bx_3)
\end{align*}
The expressions above on the right match the values of $({\hat{\sf K}}_a \sqcup {\hat{\sf K}}_b)(\tilde{\pi})$ following the interpretation of $\tilde{\pi}$ as an element of the free product $\mathcal{N}\mathcal{C}\sqcup \mathcal{N}\mathcal{C}$.
To obtain the operator-valued free cumulants of $ab$, seen as multilinear maps on $B$, we set the $y$'s equal to $1_B\in B$ in formula \eqref{eqn:krewerascomplementop}.
We explain how this degree reduction yields a sum over non-crossing \emph{linked} partitions in place of a sum over non-crossing partitions.

Pick a non-crossing partition $\pi \in \nc(n)$. Figure~\ref{fig:nctoncl} displays a partition $\tilde{\pi}$ with the blocks coloured according to the parity of the elements. Recall that the front gap of a block is the gap located in front of the first leg and the back gap is the one located just after the last leg. We symbolize evaluation to the unit $1_B$ of $B$ of the variables $y$'s by crosses. It is clear from the drawing in Figure~\ref{fig:nctoncl} that each block sees a cross either in its back gap or in its front gap.

Looking separately at each block of the partition $\tilde{\pi}$, one notices that a (odd) \emph{black block} has its \emph{back gap} marked with a cross whereas a (even) \emph{blue block} has its \emph{front gap} marked with a cross. 
The blue outer block plays a special role, as will become clear further below. We notice that this blue outer block is trivial (which means here a singleton) if the partition $\pi$ is irreducible non-crossing.

\begin{figure}[!ht]
		\includesvg[scale=1.2]{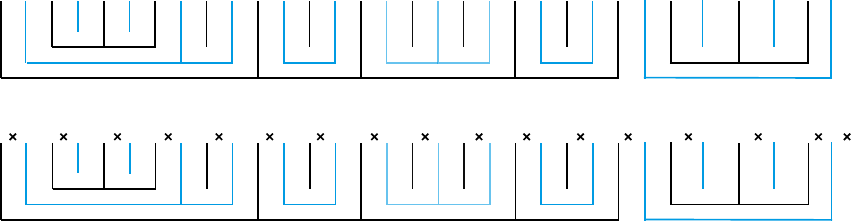}
	\caption{\label{fig:nctoncl} In the upper half, we have a partition $\tilde{\pi}$, made from the non-crossing partition $\pi$, drawn in black, and its Kreweras complement, drawn in blue. Below, we symbolized with a cross evaluations to $1_B$ of a variable in $B$ (that fall in a gap) }
\end{figure}
The multilinear map
\begin{equation*}
	\label{eqn:toobtain}
B^{\otimes (n+1)}\ni (x_{0},\ldots,x_{n}) \mapsto ({\sf \hat{K}}_{a} \sqcup {{\sf \hat{K}}_{b}})(\tilde{\pi})(x_{0},1_B,x_{1},1_B,\ldots,x_{n-1},1_B,x_n)
\end{equation*}
associated to the partition $\tilde{\pi}$ after evaluation of the $y$ variables to $1_B$ can thus be obtained by composing in the operad $\mathrm{End}_{B}$ the following multilinear maps
\begin{align*}
	 &k_{L}^{a}({n})(b_{1},\ldots,b_{n}) = \kappa_{n}(b_{1}a,\ldots,b_{n}a),\\ &k^{b}_{R}(n)(b_{1},\ldots,b_{n}) = \kappa_{n}(bb_{1},\ldots,b)b_{n}, &&k^{b}_{LR}(n+1)(b_{0},\ldots,b_{n}) = b_0\kappa_{n}(bb_{1},\ldots,b)b_{n}\quad
\end{align*}
with $b_0,\ldots,b_n \in B$.
How these multilinear maps are composed together to obtain \eqref{eqn:krewerascomplementop} is best understood by associating to $\tilde{\pi}$ a non-crossing linked partition $\tilde{\pi}_{\ell}$. First, we choose a tree monomial $\tau(\tilde{\pi})$ representing $\tilde{\pi}$. We explained that $\tilde{\pi}$, with the blocks coloured, should be seen as an element of the free product $\mathcal{N}\mathcal{C} \sqcup \mathcal{N}\mathcal{C}$. Hence, the tree monomial $\tau(\tilde{\pi})$ representing $\tilde{\pi}$ has coloured corollas, too. We can impose on $\tau(\tilde{\pi})$ the following requirements:
\begin{enumerate}[\indent 1.]
    \item \label{item:first} the root corolla is blue and its rightmost input is a leaf of $\tau(\tilde{\pi})$, 
    \item \label{item:second} the leftmost and rightmost inputs of a corolla are leaves of $\tau(\tilde{\pi})$.
\end{enumerate}
Owing to the very definition of the Kreweras complement of a non-crossing partition, the restriction of $\tilde{\pi}$ to an interval of integers bounded by two legs of a block of $\tilde{\pi}$ is irreducible. This entails property $\ref{item:second}$ for any tree monomial representing $\tilde{\pi}$.

Next we erase leaves from the tree monomial $\tau(\tilde{\pi})$ according the following rules:
\begin{enumerate}[\indent 1.]
    \item \label{item:stepone} we erase the rightmost leaf of a black corolla,
    \item \label{item:steptwo} we erase the leftmost leaf of a blue corolla different from the root,
		\item \label{item:stepthree} finally, from the set of corollas obtained by applications of the two above points, we erase corollas with a single input, excepted the root corolla.
\end{enumerate}
By definition, the tree monomial $\tau(\tilde{\pi}_{\ell})$ always has a blue root corolla.
\begin{figure}[!ht]
	\includesvg[scale=0.8]{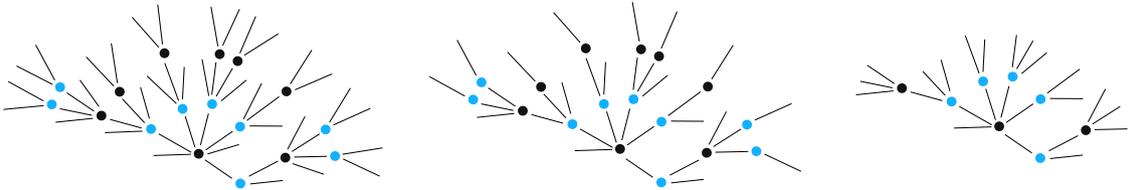}
	\caption{\label{fig:stepspitopil} On the left, the tree monomial $\tau(\tilde{\pi})$ representing the partition in Fig. \ref{fig:nctoncl}. In the center, we applied item \ref{item:stepone} and item \ref{item:steptwo}. On the right hand side we see the resulting monomial representing the non-crossing linked partition $\tilde{\pi}_{\ell}$. }
\end{figure}

We obtain a tree $\tau(\tilde{\pi}_{\ell})$ which, by making the substitution $\mathbb{1}_n \rightarrow 1_n$, $n\geq 2$ on each corolla different from the root and $\mathbb{1}_{n}\to\theta_n$ for the root corolla can be seen as a tree monomial on one-block ncl partitions representing a ncl partition $\tilde{\pi}_{\ell}$ in the nesting-or-linking operad. In Figure~\ref{nctoncl2}, we have represented the (non connected) ncl partition $\tilde{\pi}_{\ell}$ resulting from the process described above, starting with the partition pictured in Figure  \ref{fig:nctoncl}. Notice that all singletons in $\pi$ are eliminated by this process. The association $\mathcal{N}\mathcal{C} \ni \pi \mapsto \tilde{\pi}_\ell \in \mathcal{N}\mathcal{C}\mathcal{L}$ is of course not bijective for the image does only contain ncl partitions with nested blocks alternatingly coloured black or blue. Besides, a blue block of $\tilde{\pi}_{\ell}$ always has its rightmost leg $\emph{free}$ and a black block always has its leftmost block $\emph{free}$.
Finally, notice that the non-crossing linked partition $\tilde{\pi}_{\ell}$ does not meet the requirement of Dykema to have no blocks sharing their minimal elements, as shown in Figure \ref{fig:nctoncl}. This is the reason why we consider a more general definition of a non-crossing linked partition.

\begin{figure}[!ht]
		\includesvg[scale=0.75]{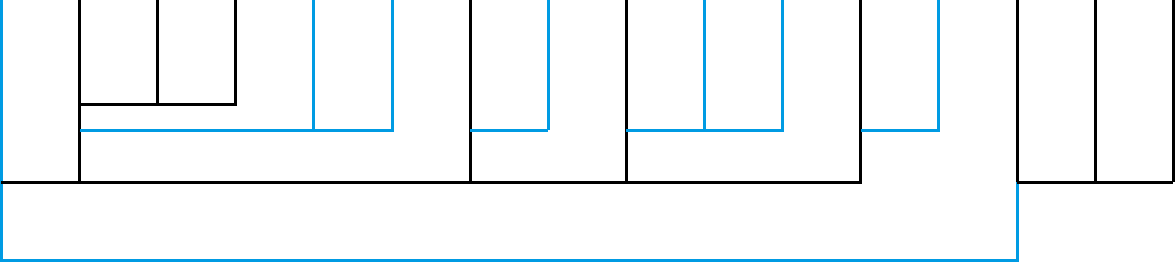}
	\caption{\label{nctoncl2} The resulting non-crossing linked partition.}
\end{figure}

To a blue root corolla of $\tau(\tilde\pi_{\ell})$ (representing the outer blue block in Figure \ref{nctoncl2}) corresponds a left and right $B$ linear map from the sequence $k^b_{LR}$, while to the other blue corollas correspond right $B$ linear maps from the sequence $k^b_{R}$. To black corollas correspond left linear multilinear maps in the sequence $k^a_L$. This entails that $\tilde{\pi}_{\ell}$ should in fact be seen as an element of the triple free product $\mathcal{NCL}^{\sqcup 3}$, with the blue outer block seen as element of the third copy of $\mathcal{N}\mathcal{C}\mathcal{L}$ in $\mathcal{NCL}^{\sqcup 3}$ and the other blocks distributed to the two remaining copies depending on their colour.
We call $\hat{\sf K}_{a}^{L}$ and $\hat{\sf K}_{b}^{R}$ the operadic morphisms on $\mathcal{NCL}$ with values in $\rm{End}_B$ that evaluate on one-block ncl partitions as $k^a_{L}$ and $k^b_R$ respectively. Finally, set for any non-crossing linked partitions $\pi_{\ell}$
$$
\hat{\sf K}_{a,b}({\pi}_{\ell})
:= (\hat{\sf K}_a^{L}\sqcup \hat{\sf K}_b^{R})({\pi}_{\ell}).
$$
If $\tilde{\pi}$ is an \emph{irreducible} non-crossing partition, the blue root corolla of $\tau(\tilde\pi)$ has only two inputs and the ncl partition $\tilde{\pi}_{\ell}$ has a singleton for outer blue block. Hence, to an irreductible non-crossing partition $\pi$ in NC$(n)$ corresponds a \emph{connected} and bi-coloured ncl partition obtained by restricting $ \tilde{\pi}_{\ell}$ to $\llbracket 1,n-1 \rrbracket$.
Define the multilinear function series $V_{a,b}$ whose homogeneous component of order $n$ is the sum of the multilinear maps ${\hat{\sf K}}_{a,b}(\tilde{\pi}_{\ell})$ with $\pi$ ranging over the set of all irreducible non-crossing partitions,
\begin{equation*}
    V_{a,b} = 1 + \sum_{n=1}^{\infty} V_{a,b}^{(n)},\quad V_{a,b}^{(n)}(b_1,\ldots,b_n)=\sum_{\pi\in{\rm NC}_{\rm irr}(n)}{{\hat{\sf K}}}_{a,b}(\tilde{\pi}_{\ell})(1,b_1,\ldots,b_n,1).
\end{equation*}
The next proposition follows from the discussion above. Recall that the series $K_a, K_b \in \mathbb{C}[[{\rm End}(B)]]$ have been defined at the beginning of this section. The two products $\centerdot$ and $\times$ are defined in Section $\ref{ssec:operadwithmult}$. In addition, we set
\begin{equation*}
	{K}_{ab} = 1 + \sum_{n\geq 1} K^{(n)}_{ab},\quad K^{(n)}_{ab}(b_1,\ldots,b_n)=\kappa_n((ab)b_1,(ab)b_2,\ldots,(ab)b_n,(ab)),~b_1,\ldots,b_n \in B.
\end{equation*}
\begin{proposition}
Let $a$ and $b$ be two free random variables in $\mathcal{A}$, then
\begin{align}
	\label{FEQ1}
	V_{a,b} = K_{a}\times [(K_{b}\times [I\centerdot V_{a,b}])\centerdot I]
\end{align}
and the cumulant series $K_{ab}$ is given by
\begin{equation}
	\label{FEQ2}
	K_{ab} = V_{a,b} \centerdot (K_{b}\times [I \centerdot V_{a,b}]).
\end{equation}
\end{proposition}

\subsection{Short proof of the twisted factorization of the $T$-transform}
\label{ssec:shortprooftwistedmultiplicativity}

In this subsection, we give a short graphical proof for the twisted factorization of the $T$-transform \eqref{eqn:twisted}. Following the work of Dykema, we define two subsets of multilinear function series,
\begin{equation*}
	\mult[[B]]_{0} = \{A \in \mult[[B]] : A_{0} = 0 \},\quad G=\mult[[B]]_{1} = \{A \in \mult[[B]] : A_{0} = 1 \}.
\end{equation*}
We represent by planar rooted trees the operations of \emph{concatenation} (the product $\centerdot$) and \emph{composition} (the product $\times$) of multilinear function series.
The composition $A \times D$ of two series $A \in \mult[[B]]$ and $D \in \mult[[B]]_{0}$ by a two nodes graph with a single vertical edge.

Doing so, we associate to the multilinear function series $A \in \mult[[B]]$ an operator with a single input acting on $\mult[[B]]_{1}$,
\begin{equation}
	\mult[[B]]_{1} \ni D \mapsto A \times D.
\end{equation}
The above operator is a set operator, in particular it is not linear. The resulting series belongs either to $\mult[[B]]_{0}$ if $A \in \mult[[B]]_{0}$ or to $\mult[[B]]_{1}$ if $A \in \mult[[B]]_{1}$.

More generally, if $W=E_{1} E_{2} \cdots E_{n}$ is a \emph{word} on multilinear function series, we associate to $W$ set operators, with multiple inputs, acting on $\mult[[B]]_{1}$. Each of these operators is drawn as a corolla with its root decorated by $W$ and with at most $n$ leaves. Each leaf corresponds to composition of the multilinear function series with one of the letters $E_{i}$, followed by concatenation of the resulting multilinear function series. For example, in the case $W = E_{1}E_{2}$, we have drawn in Figure~\ref{fig:operators} the associated operators.

\begin{figure}[!ht]
	\scalebox{0.9}{
		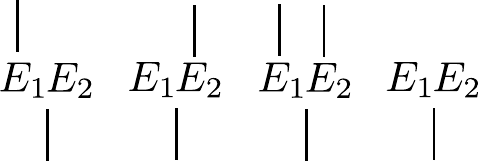}
	\caption{\label{fig:operators} Operators associated with the word $E_{1}E_{2}$, from left to right, $A \mapsto (E_{1}\circ A)E_{2},~A \mapsto E_{1} (E_{2}\circ A), (A,D)\mapsto (E_{1}\circ A)(E_{2}\circ D)$ and $E_{1}\centerdot E_{2}$.  }
\end{figure}

Notice that the edges of the corollas drawn in Figure \ref{fig:operators} should be coloured, with $1$ for the inputs and with $0$ for the output, respectively $1$, if $E_{1} \cdot \,\cdots\,\cdot E_{n}{}^{0} = 0$, respectively, $E_{1} \cdot \,\cdots\,\cdot E_{n}{}^{1} = 1$. We omit these colourizations to lighten notations.

In Figure \ref{fig:fpcumulants}, we represent graphically the defining relation of the $T$-transform and in Figure \ref{fig:fixedpointkreweras} the two equations \eqref{FEQ1} and \eqref{FEQ2}.
\begin{figure}[!ht]
	\scalebox{0.8}{
		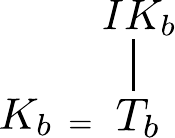
	}
	\caption{\label{fig:fpcumulants} Equation satisfied by the $T$-transform and the free cumulants. }
\end{figure}


\begin{figure}[!ht]
	\scalebox{0.7}{
		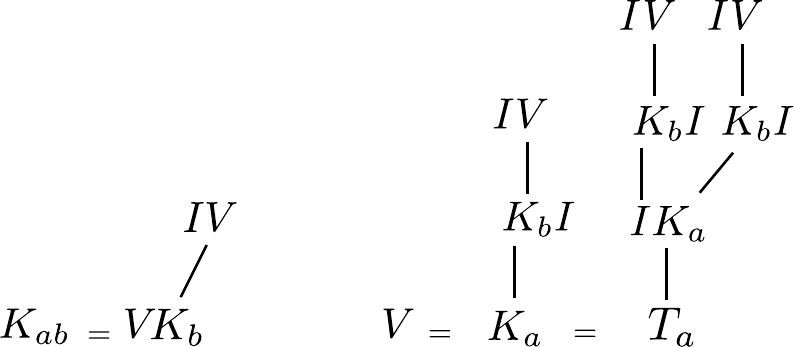
	}
	\caption{\label{fig:fixedpointkreweras} Graphical representation of equations \eqref{FEQ1} respectively \eqref{FEQ2}. Note that we omitted the indices at $V$ lighten notation. 
	}
\end{figure}

\begin{proposition}[Prop.~2.3. in \cite{dykema2007multilinear}]
	\label{prop:distributivity}
	Let $A$, $C$ and $D$ ($D_{0}=0$) be three multilinear function series, then
	\begin{equation*}
		[A\centerdot C] \times D
		= (A \times D)\centerdot (C \times D).
	\end{equation*}
\end{proposition}
Notice that since $1=1 \circ E$, one has $(A\circ E)^{-1} = A^{-1} \circ E$.

\begin{theorem}[Theorem 7.18 in \cite{dykema2007multilinear}]
	\label{thm:twistedmultplicativityun}
	Let $a,b$ be two free random variables, then
	\begin{equation*}
		T_{ab} = (T_{a}\times [T_{b}\centerdot I\centerdot T_{b}^{-1}]) \centerdot T_{b}.
	\end{equation*}
\end{theorem}

\begin{proof}
	The proof of the statement is represented diagrammatically in Figure~\ref{fig:proofmultiplicativity} and Figure~\ref{fig:proofmultiplicativitydeux}. Note that we have omitted the indices at $V=V_{a,b}$ to lighten the notation.

	\begin{figure}[!ht]
		\scalebox{0.75}{
			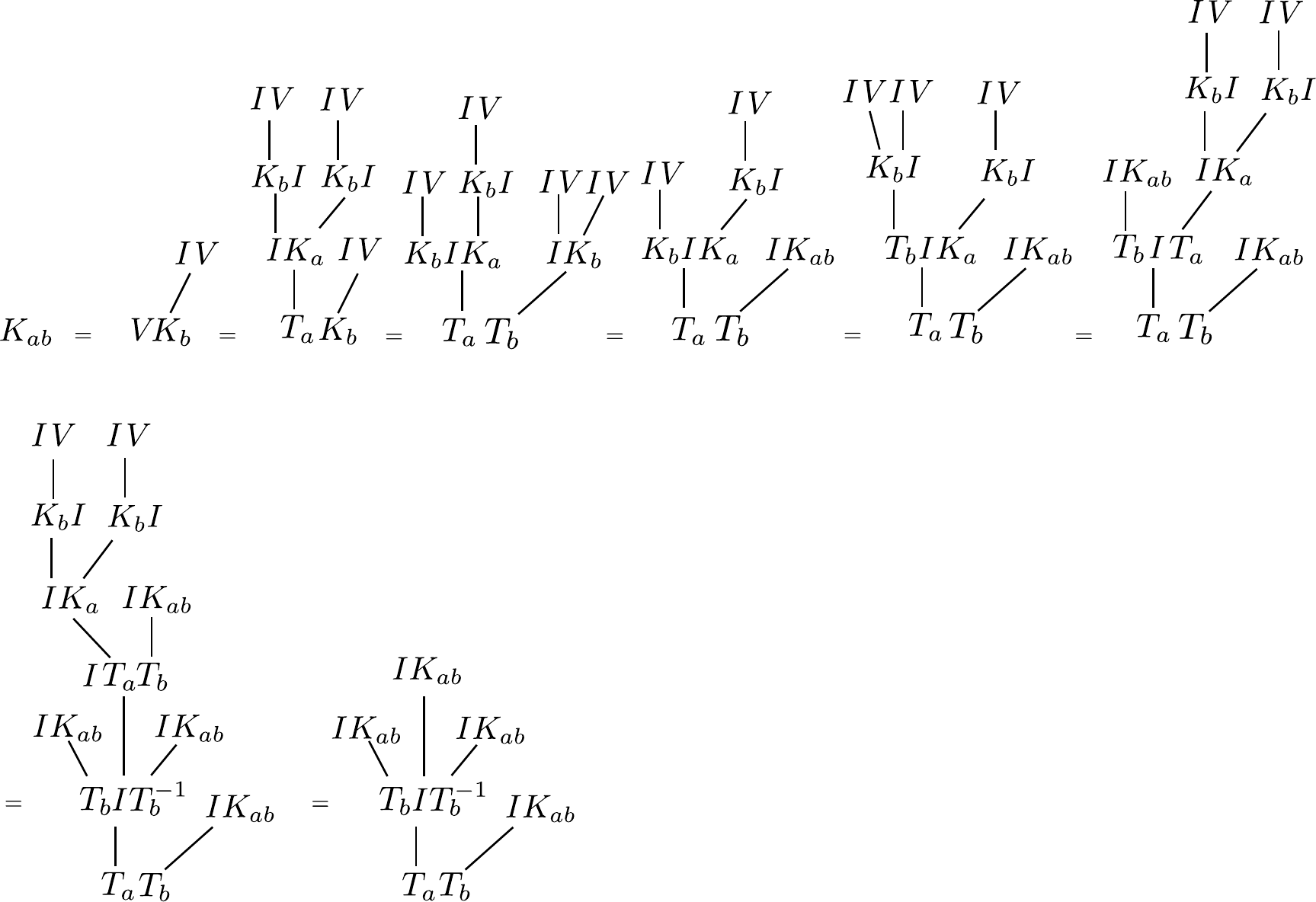
		}
		\caption{\label{fig:proofmultiplicativity} Proof of the multiplicativity property of the $T$-transform with respect to multiplication of free random variables
		}
	\end{figure}

	We detail the computations of Figure~\ref{fig:proofmultiplicativity}. For the first equality, we use equation \eqref{FEQ1} and for the second one equation \eqref{FEQ2}. The third one follows from inserting the defining equation for the $T$-transform of $b$ (see Figure \ref{fig:fpcumulants}). We then recognize the equation \eqref{FEQ1} in the leftmost tree attached to the node $T_{a}T_{b}$. The fourth and fifth equalities proceed from the same computations. To continue, we use the relation drawn in Figure \ref{fig:distributivity} for the expression circled with a dotted line in Figure~\ref{fig:proofmultiplicativitydeux}.
	\begin{figure}[!ht]
		\scalebox{0.8}{
			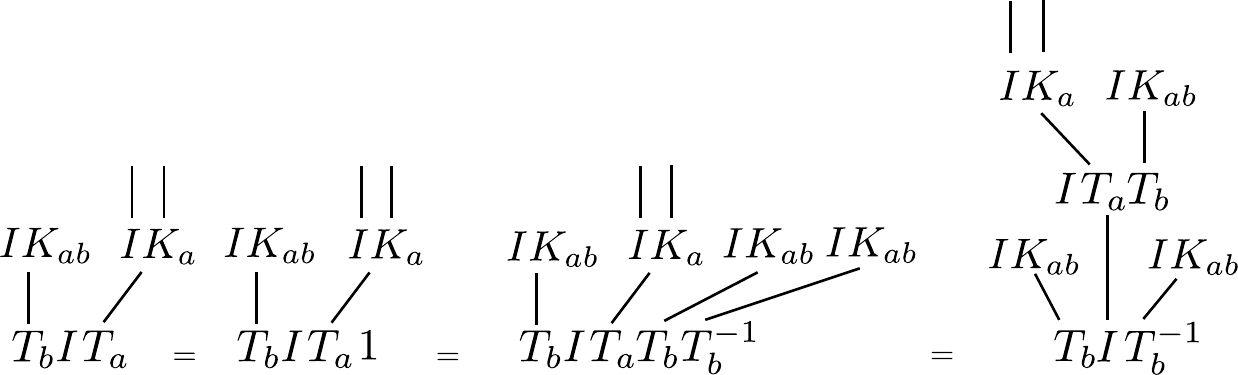
		}
		\caption{\label{fig:distributivity} Direct corollary of the relation $1 = 1 \times E = (A\times E)\centerdot(A^{-1}\times E)$.}

	\end{figure}
\begin{figure}[!ht]
	\scalebox{0.75}{
		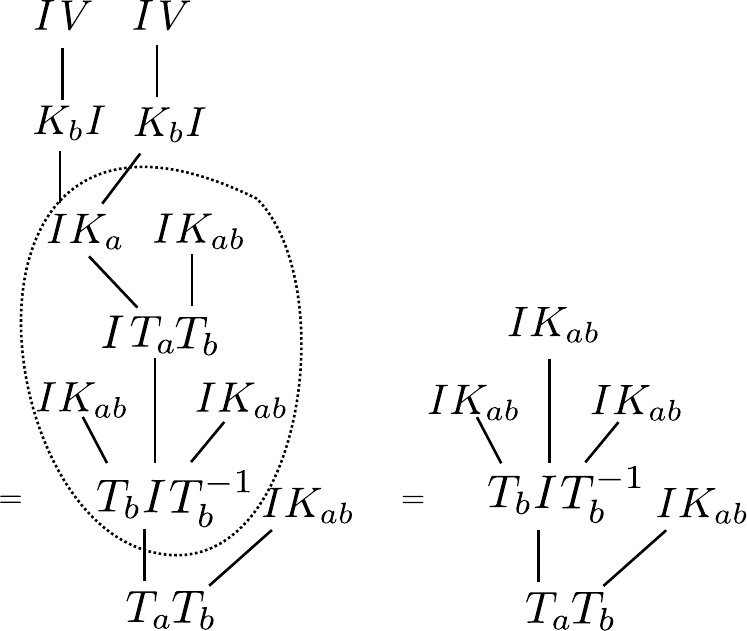
	}
	\caption{\label{fig:proofmultiplicativitydeux} End of the proof of the multiplicativity property of the $T$-transform with respect to multiplication of free random variables
	}
\end{figure}
\end{proof}

\subsection{Free product in an operad with multiplication}
\label{ssec:freeproductoperadmultiplication}
In this section, we give a more conceptual proof of the twisted factorization for the $T$-transform defined in the abstract setting of an operad with multiplication $(\mathcal{P},\gamma, m)$. This serves the objective of emphasizing the key role played by distributivity of the product $\times$ over the product $\centerdot$ (see Section \ref{ssec:operadwithmult}).%

Pick $A$ and $B$ two formal series in $\fs$ and define their concatenation product $A \centerdot B \in \fs$ by
\begin{equation}
	~A \centerdot B = \sum_{n\geq 1} \sum_{\substack{k,q\geq 0 \\ k+q = n}} A_{k} \cdot B_{q} = \sum_{n \geq 1} \sum_{\substack{k,q\geq 0 \\ k+q = n}}\{m;A_{k}\otimes B_{q}\}.
\end{equation}
The following proposition is a direct consequence of associativity of $m$ and of the operadic product $\gamma$.

\begin{proposition}
	$(\fs, \centerdot)$ is an graded  associative algebra.
\end{proposition}

Set $\fs_{0}$ equal to the unitization (for the product $\centerdot$) of $\fs$,
$$
\fs_0 = \mathbb{C}1 \oplus \fs,~{\rm deg}~1 = 0
$$
and put
$$
\gr = \{ x \in \fs_0 : x_{0} = 1 \}.
$$
The set $G^{inv}$ if endowed with the concatenation product $\centerdot$ is a group. We have introduced so far two formal groups, $G$ and $G^{inv}$. The group $G$ is sometimes called the diffeomorphism group of the operad $\mathcal{P}$. Regarding notations, a generic element of the group $G$ will be denoted $g$ and a generic element of $\gr$ will be denoted $h$.

We define next two actions, a right action $\curvearrowleft$ of the group $G$ on $G^{inv}$, compatible with the product $\centerdot$, and a left action $\curvearrowright$ of $G^{inv}$ on $G$ by conjugation. We begin with the former.

\begin{definition}[Right action $\curvearrowleft$ of $G$ on $G^{inv}$]
Pick elements $h\in \gr$ and $g \in G$ and define $1 \curvearrowleft g := 1$ and
\begin{equation}
	h \curvearrowleft g := h \times g.
\end{equation}
\end{definition}

Owing to associativity of the product $\times$, $\curvearrowleft$ is a right action of $G$. Besides, it is compatible with $\centerdot$ in the following sense.

\begin{proposition}
	Pick $h,h^{\prime} \in G^{inv}$ and $g\in G$, then
	\begin{equation}
		\label{eqn:modulelalg}
		(h \centerdot h^{\prime}) \curvearrowleft g = (h \curvearrowleft g) \centerdot( h^{\prime} \curvearrowleft g),~(h \curvearrowleft g)^{-1} = h^{-1}\curvearrowleft g.
	\end{equation}
\end{proposition}

\begin{proof}
    Let $h,h^{\prime} \in G^{inv}$ and $g\in G$. Thanks to the definition of the concatenation product $\centerdot$ and the action $\curvearrowleft$,
    \begin{align*}
        (h\centerdot h^{\prime})\curvearrowleft g &= 1 + \sum_{\substack{k,q\geq 0 \\ k+q \geq 1}} \gamma(\{m;h_{k}h^{\prime}_{q} \} \otimes g\otimes\ldots\otimes g) \\
        &= 1+\sum_{\substack{k,q\geq 0 \\ k+q \geq 1}} \gamma(\gamma (m \otimes (h_{k}\otimes h^{\prime}_{q})) \otimes (g\otimes\cdots\otimes g)) \\ &= 1+\sum_{\substack{k,q\geq 0 \\ k+q \geq 1}} \gamma(m \otimes (\gamma(h_{k}\otimes g\otimes \cdots\otimes g) \otimes \gamma(h_{q}\otimes g\otimes\cdots\otimes g))) \\
        &= (h \curvearrowleft g) \centerdot (h^{\prime} \curvearrowleft g).
    \end{align*}
This concludes the proof.
\end{proof}

\begin{definition}[Left action $\curvearrowright$ of $G^{inv}$ on $G$]
The group $G^{inv}$ acts from the left on $G$ by conjugation,
\begin{equation}
	\quad h \curvearrowright g = h \centerdot g \centerdot h^{-1},\quad h \in G^{inv},~g \in G.
\end{equation}
\end{definition}

The two actions $\curvearrowleft$ and $\curvearrowright$ are compatible in the following sense.
\begin{proposition}
	\label{prop:galgebra}
	Let $h \in G^{inv}$, $g,g^{\prime} \in G$, then 
	\begin{equation}
		\label{eqn:moduleconjugation}
		(h \curvearrowright g^{\prime}) \times g  = (h \curvearrowleft g)\curvearrowright (g^{\prime}\times g)
	\end{equation}
\end{proposition}
\begin{proof}
    Let $h \in G^{inv}$ and $g,g^{\prime} \in G$,
    \begin{align*}
        (h \curvearrowright g^{\prime}) \times g &= 1+ h\centerdot g  + g\centerdot h^{-1} \\
        &+\sum_{k,l,m\geq 1} \gamma(\gamma(m\otimes \gamma(m \otimes (h_{k}\otimes g_{l}))\otimes h^{-1}_{m}) \otimes (g \otimes\cdots\otimes g)) \\
        &\hspace{-1cm}=1+ h\centerdot g  + g\centerdot h^{-1} \\
        &\hspace{-2cm}+\sum_{k,l,m\geq 1}\gamma\Big(m^{(2)}\otimes \Big[\gamma(h_{k}\otimes (g\otimes\cdots\otimes g)) \otimes \gamma(g_{l}\otimes (g\otimes\cdots\otimes g)) \otimes \gamma(h^{-1}_{l}\otimes (g \otimes \cdots \otimes g))\Big]\Big) \\
        &\hspace{-1cm}=(h \curvearrowleft g) \centerdot (g \times g^{\prime}) \centerdot (h^{-1} \curvearrowleft g) \\
        &\hspace{-1cm}= (h \curvearrowleft g) \centerdot (g \times g^{\prime}) \centerdot (h\curvearrowleft g)^{-1} \\
        &\hspace{-1cm}= (h \curvearrowleft g)\curvearrowright (g^{\prime}\times g).
    \end{align*}
\end{proof}
The identity $I$ of the operad $(\mathcal{P}, \gamma)$ induces two injections of the set $G^{inv}$ into $G$, obtained by left translation and right translation by $I$,
\begin{equation*}
	\label{eqn:translations}
	\begin{array}{clll}
		\lambda: & \gr & \rightarrow & G, \\
		       & h                     & \mapsto     & I\centerdot h
	\end{array},
	\begin{array}{clll}
		\rho: & \gr & \rightarrow & G \\
		       & h                     & \mapsto     & h\centerdot I
	\end{array}
\end{equation*}
The maps $\rho$ and $\lambda$ are not group morphisms. However, they satisfy a number of interesting properties listed below:
\begin{enumerate}[\indent 1.]
\item \label{item:inj} both maps are injective,
\item \label{item:ranges} their respective ranges are stable with respect to the product $\times$,
\item \label{item:coalg} they extend to coalgebra morphisms between the bialgebras of polynomial functions on $G^{inv}$ and $G$,
\begin{equation*}
	(\lambda \otimes \lambda) \circ \Delta^{G^{inv}} = \Delta^{G} \circ \lambda,~ (\rho \otimes \rho) \circ \Delta^{G^{inv}} = \Delta^{G} \circ \rho
\end{equation*}
with $\Delta^{G}(g) = g\otimes g$ and $\Delta^{G^{inv}}(h) = h \otimes h$.
\end{enumerate}

The range of $\rho$ and $\lambda$ are subgroups of $G$ that are usually denoted $G^{\rho}$, respectively $G^{\lambda}$ in the literature, see \cite{frabetti}.
Next, owing to Proposition \ref{prop:galgebra} and item \ref{item:coalg} above, $G^{inv}$ can be endowed with two additional products,
\begin{equation*}
	h \star_{l} h^{\prime} = h^{\prime} \centerdot (h\curvearrowleft \lambda(h^{\prime})),~ h \star_{r} h^{\prime} = (h \curvearrowleft \rho(h^{\prime}))\centerdot h^{\prime}
\end{equation*}

\begin{proposition}
\label{prop:cocycle}
	Let $h,h^{\prime} \in G^{inv}$ and $g \in G$,
	\begin{align}
	\label{eqn:cocyles}
		 & \lambda(h) \times \lambda(h^{\prime}) = \lambda(h \star_{l} h^{\prime}),\quad \rho(h) \times \rho(h^{\prime}) = \rho(h\star_{r}h^{\prime})
	\end{align}
	\begin{align}
	\label{eqn:rholambda}
		 & \rho(h) = h \curvearrowright \lambda(h)
	\end{align}
\end{proposition}
\begin{proof} With $h,h^{\prime} \in G^{inv}$, one has
	\begin{equation*}
		\lambda(h) \times \lambda(h^{\prime})= (I\centerdot h)\times (I\centerdot h^{\prime}) = (I\centerdot h^{\prime}) \centerdot (h\times (I\centerdot h^{\prime})) = I\centerdot(h^{\prime} \centerdot (h\times (I\centerdot h^{\prime})))).
	\end{equation*}
	The computations with $\rho$ in place of $\lambda$ are similar. The second statement is obvious.
\end{proof}

The above proposition can be restated by saying that $\lambda^{-1}$ and $\rho^{-1}$ are two cocycles with respect to the right action $\curvearrowleft$ and the group product on $G^{\lambda}$, respectively $G^{\rho}$,
$$
\lambda^{-1}(g\times g^{\prime}) = \lambda^{-1}(g)\centerdot\lambda^{-1}(g^{\prime})\curvearrowleft g^{\prime}.
$$

\begin{proposition} Pick $h$ and $h^{\prime}$ in $\gr$, then
	\begin{equation}
		\label{eqn:ilir}
		a)~\rho(h)\times\lambda(h^{\prime}) = (h^{\prime})^{-1} \curvearrowright \rho(h\star_{l} h^{\prime}),\quad
		b)~\lambda(h) \times \rho(h^{\prime}) = h^{\prime} \curvearrowright \lambda(h \star_{r} h^{\prime}),\quad
	\end{equation}
\end{proposition}
\begin{proof} We rely on the very definition of the left action $\curvearrowleft$ and of the maps $\rho$ and $\lambda$. Let $h,h^{\prime} \in \gr$, then
	\begin{equation}
		\lambda(h) \times \rho(h^{\prime}) = h^{\prime} \centerdot I \centerdot (h \times (h^{\prime}\centerdot I))  = h^{\prime} \centerdot (I \centerdot (h\times (h^{\prime}\centerdot I) ) \centerdot h^{\prime} ) \centerdot (h^{\prime})^{-1}
	\end{equation}
\end{proof}

We remark that relation $b)$ in the last proposition can be seen to be implied by Proposition \ref{prop:galgebra} and the cocycle property for $\rho$. However, it does not seem to be a consequence of the cocycle property for $\lambda$ together with Proposition \ref{prop:galgebra} and (\ref{eqn:rholambda}).

Next, we define the free product of two elements in $G^{inv}$ in the abstract setting of an operad with multiplication.
\begin{definition}[Free product]
	Pick $k_{a}$ and $k_{b}$ in $G^{inv}$ and define the free product $k_{a} \boxtimes k_{b} \in G^{inv}$ as $k_{a} \boxtimes k_{b} := k_{b} \star_{l} v$, where
	\begin{align*}
		 v := k_{a} \curvearrowleft \rho(k_{b} \curvearrowleft \lambda(v))
	\end{align*}
\end{definition}

\begin{theorem}
	\label{thm:multttransform}
	Pick $k_{a}$ and $k_{b}$ in $G^{inv}$ and assume the following fixed point equations in $G^{inv}$ to hold
	\begin{equation}
		k_{a} = t_{a} \curvearrowleft \lambda(k_{a}),\quad k_{b} = t_{b} \curvearrowleft \lambda(k_{b}),
	\end{equation}
	with $t_{a}$ and $t_{b}$ in $G^{inv}$. One has
	\begin{equation*}
		k_{a} \boxtimes k_{b} = \big(\left[t_{a} \curvearrowleft (t_{b} \curvearrowright I)\right] \centerdot t_{b} \big) \curvearrowleft \lambda(k_{a}\boxtimes k_{b}).
	\end{equation*}
\end{theorem}

\begin{proof} Set $k_{ab} = k_{a} \boxtimes k_{b}$. First, we use the fact that $\curvearrowleft$ is a right action of the group $(G,\times)$ to write
	\begin{align*}
		k_{ab} = v \centerdot (k_{b} \curvearrowleft \lambda(v)) = v \centerdot (t_{b} \curvearrowleft (\lambda(k_{b}) \times \lambda(v))) = v \centerdot (t_{b} \curvearrowleft \lambda(k_{b} \star_{l} v)) = v \centerdot (t_{b} \curvearrowleft \lambda(k_{ab})).
	\end{align*}
	We then use the fixed point equations satisfied by $k_{a}$ and $k_{b}$,
	\begin{align*}
		v = t_{a} \curvearrowleft (\lambda(k_{a}) \times \rho(k_{b} \curvearrowleft \lambda(v))) & = t_{a} \curvearrowleft (\lambda(k_{a}) \times \rho(t_{b} \curvearrowleft (\lambda(k_{b}) \times \lambda(v)))). \\
	    & = t_{a} \curvearrowleft (\lambda(k_{a})\times \rho(t_{b} \curvearrowleft \lambda(k_{b}\star_{l}v)))         \\
	     & = t_{a} \curvearrowleft (\lambda(k_{a}) \times \rho(t_{b} \curvearrowleft \lambda(k_{ab}))).
	\end{align*}
	By using equation \eqref{eqn:ilir} with $h^{\prime} = t_{b} \curvearrowleft \lambda(k_{ab})$ and $h=k_{a}$ we obtain
	\begin{align*}
		\lambda(k_{a}) \times \rho(t_{b} \curvearrowleft \lambda(k_{ab}))              & = (t_{b} \curvearrowleft \lambda(k_{ab})) \curvearrowright \lambda(k_{a} \star_{r} (t_{b} \curvearrowleft \lambda(k_{ab})).
	\end{align*}
	Coming back to the fixed point equation satisfied by $v$ and inserting the fixed point equation satisfied by $k_{b}$, we get also
	\begin{equation*}
		v = k_{a} \curvearrowleft \rho(t_{b} \curvearrowleft \lambda(k_{b} \star_{l} v)) = k_{a} \curvearrowleft \rho(t_{b} \curvearrowleft \lambda(k_{ab})).
	\end{equation*}
	This last equation implies
	\begin{align*}
		\lambda(k_{a}) \times \rho(t_{b} \curvearrowleft \lambda(k_{ab})) & = (t_{b} \curvearrowleft \lambda(k_{ab})) \curvearrowright (v \centerdot (t_{b} \curvearrowleft \lambda(k_{ab}))) \\
	            & = (t_{b} \curvearrowleft \lambda(k_{ab})) \curvearrowright \lambda(k_{ab}).
	\end{align*}
	Thus we obtain, for $v$ and $k_{ab}$,
	\begin{equation*}
		v = t_{a} \curvearrowleft \Big(\big(t_{b} \curvearrowleft \lambda(k_{ab})\big) \curvearrowright \lambda(k_{ab})\Big),
	\end{equation*}
	\begin{equation*}
		k_{ab} = \Big[t_{a} \curvearrowleft \Big( \big(t_{b} \curvearrowleft \lambda(k_{ab})) \curvearrowright \lambda(k_{ab})\Big)\Big] \centerdot \Big[t_{b} \curvearrowleft \lambda(k_{ab})\Big].
	\end{equation*}
	It follows from equations \eqref{eqn:moduleconjugation} and \eqref{eqn:modulelalg} that
	\begin{align*}
		k_{ab} &= \Big[t_{a} \curvearrowleft \Big((t_{b} \curvearrowright I) \times \lambda(k_{ab})\Big)\Big] \centerdot \Big[t_{b} \curvearrowleft \lambda(k_{ab})\Big]\\
		& = \Big[\Big(t_{a} \curvearrowleft (t_{b} \curvearrowright I)\Big)\curvearrowleft \lambda(k_{ab})\Big] \centerdot \Big[t_{b} \curvearrowleft \lambda(k_{ab})\Big] \\                   & = \Big((t_{a}\curvearrowleft (t_{b} \curvearrowright I)) \centerdot t_{b} \Big) \curvearrowleft \lambda(k_{ab}).
	\end{align*}
\end{proof}
If one chooses for the operad with multiplication $\mathcal{P}$ the endomorphism operad of $B$, with $m$ equal to the product in $B$, the above argument gives another proof of the twisted multiplicativity for the $T$-transform in operator-valued free probability.

\bibliographystyle{plain}
\bibliography{notesartdykema}
\end{document}